\documentclass[12pt]{article}
\RequirePackage{amsthm,amsmath, amssymb}
\usepackage{url}
 \usepackage{epsfig}
 \usepackage{graphicx}
 \usepackage{pict2e}
 \usepackage{bm}
 \newtheorem{Lemma}{Lemma}
 \newtheorem{Proposition}[Lemma]{Proposition}
 
 \newtheorem{Theorem}[Lemma]{Theorem}
 \newtheorem{Conjecture}[Lemma]{Conjecture}

 \newtheorem{Corollary}[Lemma]{Corollary}

 \newcommand{\FF}{\mbox{${\mathcal F}$}} 
  \newcommand{\CC}{\mbox{${\mathcal C}$}}
 \renewcommand{\AA}{\mbox{${\mathcal A}$}}
  \newcommand{\MM}{\mbox{${\mathcal M}$}}
  
    \newcommand{\ZZ}{\mbox{${\mathcal Z}$}}

 \newcommand{\sfrac}[2]{{\textstyle\frac{#1}{#2}}}

 \newcommand{\Reals}{{\mathbb{R}}}
  
 \newcommand{\Complex}{{\mathbb{C}}}

 \newcommand{\bC}{{\mathbf C}}
 \newcommand{\bc}{{\mathbf c}}
 \newcommand{\origin}{{\mathbf 0}}

  \newcommand{\bZ}{\mathbf{Z}}

    \newcommand{\eps}{\varepsilon}
  \newcommand{\Ex}{{\mathbb E}}
\renewcommand{\Pr}{{\mathbb P}}
\newcommand{\indic}{1}

\newcommand{\Tcouple}{T_{\mathrm{coal}}}
\newcommand{\Icouple}{I_{\mathrm{coal}}}

\newcommand{\Zcouple}{Z_{\mathrm{coal}}}
\DeclareMathOperator{\diam}{diam}
\DeclareMathOperator{\area}{area}
\DeclareMathOperator{\disc}{disc}

\DeclareMathOperator{\ave}{ave}

\DeclareMathOperator{\Leb}{Leb}
\DeclareMathOperator{\parent}{parent}
\DeclareMathOperator{\ancestor}{ancestor}
\DeclareMathOperator{\descend}{Descend}
 
 \newcommand{\Oexp}{O_{\scriptscriptstyle exp}}

\begin{document}

\title{Random partitions of the plane via Poissonian coloring, 
and a self-similar process of coalescing planar partitions}

 \author{David J. Aldous
 \thanks{Department of Statistics,
 367 Evans Hall \#\  3860,
 U.C. Berkeley CA 94720;  aldous@stat.berkeley.edu;
  www.stat.berkeley.edu/users/aldous.  Aldous's research supported by
 N.S.F Grant DMS-1504802. }}

 \maketitle

\begin{abstract}
Plant  differently colored points in the plane; then let random points (``Poisson rain") fall, and give each new point the color of the nearest existing point.
Previous investigation and simulations strongly suggest that the colored regions converge (in some sense) to a random partition of the plane.
We prove a weak version of this, showing that normalized empirical measures converge to Lebesgue measures on a random partition into measurable sets.  
Topological properties  remain an open problem.  In the course of the proof, which heavily exploits  time-reversals, we encounter a novel self-similar process of coalescing planar partitions.  
In this process, sets $A(z)$ in the partition are associated with Poisson random points $z$, and the dynamics are as follows.
 Points are  deleted randomly at rate $1$; when $z$ is deleted, its set $A(z)$ is adjoined to the set $A(z^\prime)$ of the nearest other point $z^\prime$.
 \end{abstract}



\section{Introduction}
\label{sec:int}
The work in this paper has  several motivations. 
We focus below on the most concrete motivation; more broadly, as indicated in 
sections \ref{sec:background}  and \ref{sec:other_coal}, 
we will encounter a kind of spatial analog of well-studied non-spatial models of stochastic fragmentation (in forward time) or stochastic coalescence (in reversed time).
A minor variant of the process below has been considered independently by several researchers 
(see section \ref{sec:background}), but without any published results.

As the ``elementary" variant\footnote{We mean that the model {\em definition} is elementary.}, 
choose $k \ge 2$ distinct points $z_1, \ldots, z_k$ in the unit square, and assign to point $z_i$ the color $i$ from a palette of $k$ colors.  
Take i.i.d. uniform random points $U_{k+1}, U_{k+2}, \ldots$ in the unit square, 
and inductively, for $j \ge k+1$,
\begin{quote}
give point $U_j$  the color of the closest point to $U_j$ 
amongst $U_1,\ldots,U_{j-1}$
\end{quote}
where we interpret $U_i = z_i, 1 \le i \le k$   
(there is a unique closest point a.s.; throughout the paper we omit the ``a.s." qualifier where no subtlety is involved).
This defines a process 
$S_n = (S_n(i), 1 \le i \le k)$, where $S_n(i)$ is the set of color-$i$ 
points amongst $(U_j, 1 \le j \le n)$.  
Simulations (see Figure \ref{fig_1})\footnote{Figures  \ref{fig_1} and  \ref{fig_2} created by  Weijian Han.} 
and intuition strongly suggest that there is (in some sense) an $n \to \infty$ limit which
is a random partition of the square into $k$ colored regions. 
Simulations (see Figure \ref{fig_2}) also suggest that the boundaries between these limit regions should be fractal, in some sense, though 
intuition is less clear here (see section \ref{sec:2heuristics}).

\begin{figure}
\caption{A realization within part of the unit square.  
Line segments indicate parent-child relation.}
\label{fig_1}
\begin{center}
 \includegraphics[width=130mm]{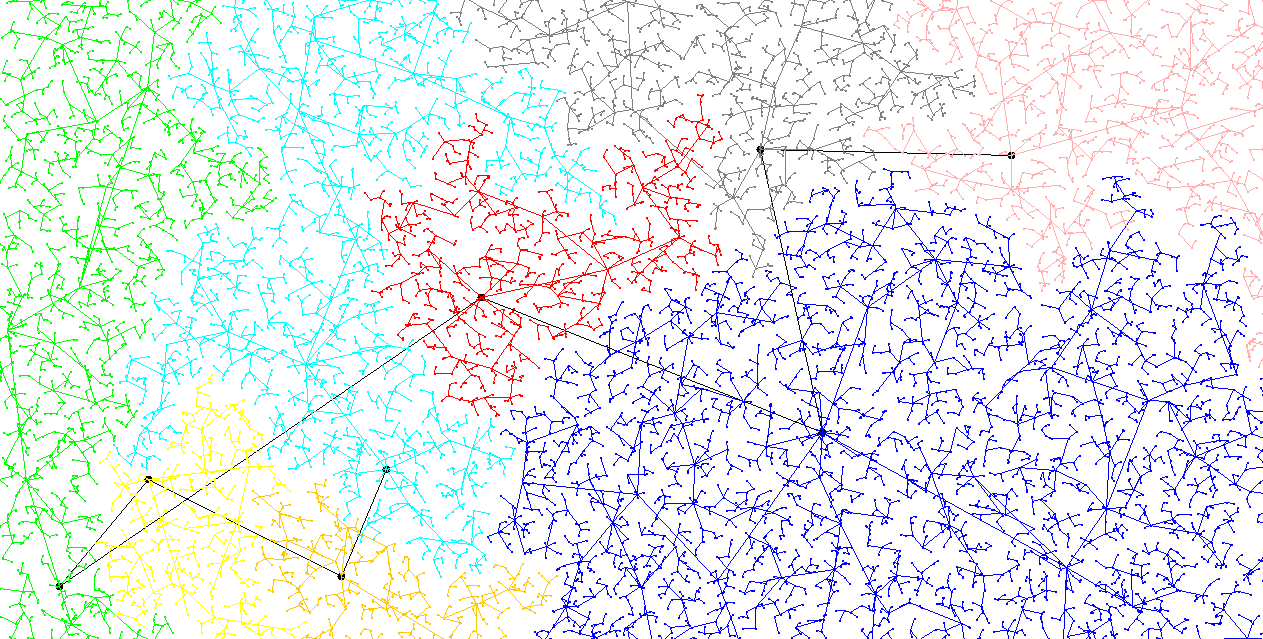}
 \end{center}

\end{figure}

\begin{figure}
\caption{A close-up of the boundary between partition components suggests the boundary is fractal.}
\label{fig_2}
\begin{center}
 \includegraphics[width=130mm]{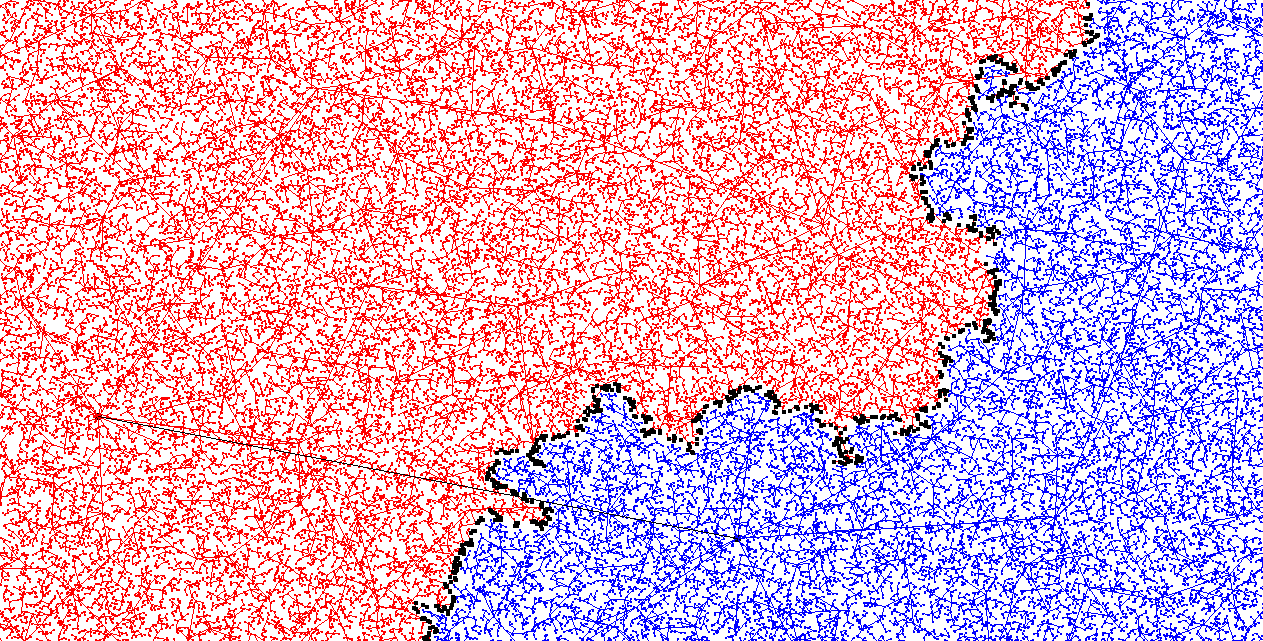}
 \end{center}

\end{figure}

\medskip
What can we actually prove?
For rigorous study, it is more convenient to consider a slightly more sophisticated model. 
On the infinite plane $\Reals^2$ and the infinite time interval $- \infty < t < \infty$ 
there is a space-time  Poisson point process (PPP), which we will envisage as the times and positions of arriving particles, 
such that the set of particles which arrive before time $t$ forms 
a spatial PPP on $\Reals^2$ with intensity $e^t$ per unit area.  
Within this process (more details and notation for what follows in this section will be given next in section \ref{sec:outline}), imagine assigning a different color to each particle present at time $t_1$, 
and then as $t$ increases suppose we color each newly arriving particle by the previous rule, that is by copying the color of the nearest existing particle. 
Intuitively, what we see in the unit square within this model, at large times $t$, must be similar
(up to boundary effects) as in the elementary model with a Poisson($e^{t_1}$) number of initial particles and with $n \approx  e^t$ total particles.

The advantage of this more sophisticated model is that we can exploit the exact self-similarity property of the underlying space-time PPP. 
In particular, by reversing time  the ``line of descent" by which a particle acquires its color from previous particles can be studied.
Moreover, suppose the first intuitive suggestion is true.
That is, after assigning different colors to particles at positions $z$ at time $t_1$, suppose
there is a $t \to \infty$ limit random partition $\AA(t_1)$ of $\Reals^2$ into regions $A(t_1,z)$ occupied by particles with the color of the particle at $z$ at time $t_1$.
This (supposed) partition valued process  $(\AA(t_1),  \infty > t_1 > - \infty)$, has a simple intuitive description in reversed time. 
Given $\AA(t_1)$ and the time-$t_1$ particle positions $z$, we obtain $\AA(t_1 - dt)$ by the rule
\begin{quote}
delete each particle with probability $dt$; for each deleted particle, at position $z$ say, let $z^\prime$ be the nearest other particle
position, and replace $A(t_1,z^\prime)$ by $A(t_1,z^\prime) \cup A(r_1,z)$.
\end{quote}
The purpose of this paper is to prove two intertwined results; 
that the random partition $\AA(t_1)$ does exist as a certain type of limit of the coloring process 
(Theorem \ref{T:limit}); 
and that the resulting reversed-time process $(\AA(t_1) : \infty > t_1 > -\infty)$ is a self-similar version of the process defined by the rule above
(Theorem \ref{T:process}).

\subsection{Notation and more detailed outline}
\label{sec:outline}
Write $\Reals \times \Reals^2$ for the set with elements $(t,z)$, 
interpreted as ``time" $t \in \Reals$ and ``position" $z \in \Reals^2$. 
Write $\bm{\Xi}$ for the Poisson point process on  $\Reals \times \Reals^2$ with mean measure 
$e^{t} dt dz$.
All the random objects considered in this paper will be constructed from $\bm{\Xi}$.
We write a typical ``point" of $\bm{\Xi}$ as $\xi = (t_\xi, z_\xi)$ or $\zeta = (t_\zeta, z_\zeta)$.
We consider $\xi$ as the label for an immortal particle with arrival or ``birth" time $t_\xi$ at position $z_\xi$, and so
\[ \Xi_{\le t} := \{\xi \in \bm{\Xi} : \ t_\xi \le t\} \]
denotes the set of particles which are alive at time $t$. 
Define $\Xi_{<t}$ analogously.
Write
\[ \ZZ_{\le t} = \{ z_\xi: \xi \in \Xi_{\le t} \} \]
for the {\em positions} of the particles at time $t$.
Of course  $\ZZ_{\le t}$ and $\ZZ_{<t}$ are Poisson point processes on $\Reals^2$ with rate $e^t$, that is mean measure $e^t dz$, 
because $\int_{-\infty}^t e^s \ ds = e^t$.  
The self-similarity properties of the PPP -- that $\ZZ_{\le t_1}$ is distributed as a spatial rescaling of $\ZZ_{\le t_2}$ -- will 
extend to self-similarity for  the process $(\AA(t_1) : \infty > t_1 > -\infty)$ 
outlined in the previous section.

To each particle $\xi$ let us assign a {\em parent} particle $\zeta = \parent(\xi)$, defined as the particle in $\Xi_{t_\xi -}$ 
for which the Euclidean distance $|| z_\zeta - z_\xi||$ is minimized.
This defines a (genealogical) 
{\em tree process}.
So for each particle $\xi$ there is an {\em ancestral sequence} of particles, written
$(\parent[i,\xi], i \ge 0)$, defined by $\parent[0,\xi] = \xi$ and then recursively by 
\[ \parent[i+1,\xi] = \parent( \parent[i, \xi]), \ i \ge 0 . \]
The associated {\em line of descent} indicates the ancestor of $\xi$ at each time $t < t_\xi$, that is
\begin{equation}
 \ancestor(t,\xi) = \parent[i,\xi] 
\mbox{ for $i \ge 1$ such that } t_{\parent[i,\xi] } \le t <  t_{\parent[i-1,\xi] }  
\label{def:ancestor}
\end{equation}
where for completeness we define 
\[  \ancestor(t,\xi) = \xi \mbox{ for } t \ge t_\xi  . \]
The first part of the proof (Proposition \ref{P1} in section \ref{sec:part1}) shows
 that for a typical particle $\xi$ present at time $0$, the distance to $\ancestor(-t,\xi)$, the ancestor at time $-t$,
is of order $e^{t/2}$, which is the same order as the distance to the {\em nearest} particle present at time $-t$.
In the second part of the proof (section \ref{sec:joint}) we first consider 
two particles present at time $0$ and distance $r$ apart.  Their lines of descent merge at some random past time $-T_r$, 
and we need an upper bound (Proposition \ref{Pcouple}) on the tail of the distribution of $T_r$.  
The methods in these sections are very concrete -- calculations and bounds involving  Euclidean geometry and spatial Poisson processes -- 
though rather intricate in detail.

The limit result we seek involves descendants (rather than ancestors) of typical particles,
and we set up notation as follows.\footnote{For ancestor-descendant pairs we systematically write $\zeta$ for the ancestor and $\xi$ for the descendant.}  
For $t_1 \le t_2$ and $\zeta \in \Xi_{\le t_1}$ define 
\begin{equation} 
\descend(t_1,t_2,\zeta):= 
\{\xi \in \Xi_{\le t_2}: \ \ancestor(t_1, \xi) = \zeta\} .
\label{def:descend}
\end{equation}
This is the set of particles born before $t_2$ whose time-$t_1$ ancestor in the line of descent was $\zeta$. 
In the coloring story, this is ``the set of particles at time $t_2$ which have inherited the same color 
as $\zeta$, if we gave all the particles at $t_1$ different colors". 
Then, still for $t_1  \le t_2$ and $\zeta \in \Xi_{\le t_1}$, define 
\begin{eqnarray}
 \mbox{ $\mu_{t_1,t_2,\zeta}$ 
is the measure $\mu$ putting weight $e^{-t_2}$} \nonumber\\
\mbox{  on the position of each particle in $\descend(t_1,t_2,\zeta)$.} \label{def:mu}
\end{eqnarray}
So $\mu_{t_1,t_2,\zeta}$ is a random  element of the space  
$\MM(\Reals^2)$ of finite measures on $\Reals^2$, equipped with the usual topology of weak convergence. 
To obtain the limit theorem  we first show (Proposition \ref{P:coupledist} in section \ref{sec:MPPs})   that there exist $t_2 \to \infty$ limits in probability (as $\MM(\Reals^2)$-valued random variables); that is, there exist random measures  $\{ \mu_{t_1,\infty,\zeta}:  \    \zeta \in  \Xi_{\le t_1}  \} $ such that
\begin{equation}
 \mu_{t_1,t_2,\zeta} \to \mu_{t_1,\infty,\zeta} \mbox{ in probability as } t_2 \to \infty, \quad (\forall \zeta \in  \Xi_{\le t_1} ) . 
 \label{mtt12}
 \end{equation}
 The proof essentially relies on Proposition \ref{P1} and self-similarity.
 We then use Proposition  \ref{Pcouple} to show that a limit $ \mu_{t_1,\infty,\zeta}$ is in fact Lebesgue measure restricted to some random set 
 $A(t_1,\zeta)$, implying that the collection $\{A(t_1,\zeta): \ \zeta \in \Xi_{\le t_1} \}$ is necessarily  a partition of $\Reals^2$.
 
 For fixed $t$ we can regard 
\[ \bZ^{(t)} = 
\{ (z_\xi, A(t,\xi)): \ \xi \in  \Xi_{\le t}  \}
\]
as a marked point process.
As $t$ increases, the process $(\bZ^{(t)}, - \infty < t < \infty)$ evolves in a way one can describe qualitatively:
\begin{quote}
new points arrive randomly at rate $e^t$ per unit area per unit time; 
when a point $\xi$ arrives at time $t$, the region $A(t,\zeta)$ associated with the closest existing point $\zeta$ is split into 
two regions $A(t +dt, \zeta)$ and $A(t +dt, \xi)$.
\end{quote}
But the probability distribution over possible splits depends on $\bZ^{(t)}$ in some complicated way which we
are unable to describe explicitly.
 
However, the key feature of this process is that as $t$ {\em decreases} the regions merge according to the simple rule stated earlier. 
To summarize:
\begin{Theorem}
\label{T:process}
The space-time PPP $\{ (t_\xi, z_\xi): \xi \in \bm{\Xi} \}$ 
can be extended to a process 
$\{ (t_\xi, z_\xi, A(t, \xi), t \ge t_\xi) : \xi \in \bm{\Xi} \}$ 
with the following properties. \\
(a) For each $- \infty < t < \infty$ the collection 
$\{ A(t,\xi) : \xi \in \Xi_{\le t} \}$ 
is a random  partition of $\Reals^2$ into measurable sets. \\
(b) The distribution  of the entire time-varying process 
$\{ (t_\xi, z_\xi, A(t, \xi), t \ge t_\xi) : \xi \in \bm{\Xi} \}$ 
is invariant under the action of the Euclidean group on $\Reals^2$. \\
(c) The process whose state at time $t$ is 
$\{ (z_\xi, A(t, \xi)) : \xi \in \Xi_{\le t} \}$ 
evolves in reversed time according to the rule:
\begin{quote}
during $[t, t - dt]$, for each $\xi \in \Xi_{\le t}$  delete $\xi$ 
(that is, remove the entry  $(z_\xi, A(t, \xi))$ )
 with probability $dt$; 
for each deleted particle $\xi$, let $\zeta$ be the nearest other particle,
and set  $A(t - dt, \zeta) = A(t, \zeta) \cup A(t,\xi)$.
\end{quote}
(d) The action of the scaling map $z \to e^{-t/2}z$ on $\Reals^2$ that takes the distribution 
of $\ZZ_{\le 0}$ to the distribution of $\ZZ_{\le t}$ also takes the distribution of 
$\{ (z_\xi, A(0, \xi)): \xi \in \Xi_{\le 0} \}$ 
to the distribution of
$\{ (z_\xi, A(t, \xi)) : \xi \in \Xi_{\le t} \}$.
\end{Theorem}
The earlier statement (\ref{mtt12}) can now be rephrased as follows, where  
we define $\mu_{t,t_2,\xi}$ as at (\ref{def:mu}) and consider it as an $\MM(\Reals^2)$-valued random variable. 
Write  $\Lambda_A$ for Lebesgue measure restricted to $A \subset \Reals^2$.
\begin{Theorem}
\label{T:limit}
For each $\xi \in \Xi_{\le t}$ we have 
\[ \mu_{t,t_2,\xi} \to \Lambda_{A(t,\xi)} \mbox{ in probability as } t_2 \to \infty \]
where the limit random sets $\{A(t, \xi) : \xi \in \Xi_{\le t} \}$ 
comprise a process with the properties stated in Theorem \ref{T:process}.
\end{Theorem}
Note this implies that the limit random sets here and in Theorem \ref{T:process} are $\sigma(\bm{\Xi})$-measurable.

Theorem \ref{T:limit} is a formalization of the  
 ``limit colored regions exist" result described in the opening section, but this particular formalization is mathematically weak in two senses.
Our formalization via weak convergence of empirical measures means, in the original ``elementary" version, that we are ignoring 
positions of $o(n)$ size subsets of the $n$ particles.  
Second, our proof gives no information about topological properties of the 
regions
$A(t,\xi)$, only that they are measurable.
In fact because the regions $A(t,\xi)$ are identified via the Lebesgue measure they support, they are only well-defined 
up to Lebesgue-null sets of $\Reals^2$.  So, for instance, the natural question 
``is $\xi$ an element of  $A(t,\xi)$?" is not well-posed.
But it is natural to  guess that the following is true.
\begin{Conjecture}
\label{C:1}
For each $t$ one can identify the regions $\{A(t,\xi): \ \xi \in \Xi_{\le t} \}$ so that the topological boundary of each region has Lebesgue measure zero.
\end{Conjecture}
If true, we could rephrase the question above as the well-posed question 
``is each $\xi$ in the interior of  $A(t,\xi)$?", 
and we conjecture the answer is Yes. 
More interestingly, assuming Conjecture \ref{C:1} is true, it is natural to conjecture that the boundaries have some (suitably defined) 
non-random fractal dimension $1 \le d < 2$, and section \ref{sec:2heuristics}  contains heuristic discussion.
Further related remarks are in the next section.
Finally, one might expect the regions to be {\em connected} sets, but this seems incorrect --  
see section \ref{sec:notfractal}.

\subsection{Background and analogous models}
\label{sec:background}
To quote the unpublished notes \cite{jordan-wade}
\begin{quote}
The [elementary] model is described in [\cite{penrose-wade},  sec. 7.6.8, pp. 270--271], 
although we are not sure of its origins: [we]  probably first learned of the problem from Mathew Penrose in about 2003, 
while Ben Hambly [personal communication] recalls that the same problem arose elsewhere at about the same time.
\end{quote}
The context of that line of work was on-line algorithms in computational and stochastic geometry.
Separately the present author learned [personal communication] that the elementary model has been considered by
Ohad Feldheim as a spatial analog of the P\'{o}lya  urn process. 

The approach in  \cite{jordan-wade} to  the elementary model 
 is to identify colored regions in the unit square as Voronoi regions, that is the set of points for which the nearest particle has a given color.
 Then via the Hausdorff metric on closed sets, it makes sense to ask whether our notion of convergence of empirical measures can be strengthened to include convergence of Voronoi  regions. 
In our language and model, this could only  be true if Conjecture \ref{C:1} is true. 
 Arguments in  \cite{jordan-wade} focus on the  length  $\ell_n$ of the boundary between the two regions (for two colors and $n$ particles in the unit square).
 Using arguments with a more geometric flavor than ours, they 
  raise and discuss  the question of whether $\ell_n = O(n^{d/2})$ for some $d<2$.  
  This mirrors our ``fractal dimension" question, and indeed would imply that Conjecture \ref{C:1} is true.   
  The arguments in this paper make surprisingly little use of the ``local geometry" of the PPP, so
 one can hope that our results might be combined with  more geometric arguments to make further progress.
 
 Note also that, intuitively, the area of the Voronoi region of a given color should behave almost as a martingale, because a 
 new particle near the boundary seems equally likely to make the area larger or smaller. 
 If one could bound the martingale approximation well enough to establish a.s. convergence of such areas, the results of this paper would follow rather trivially.
 But doing so seems to require detailed knowledge of the geometry of the boundary.

 The author's own interest in the model arose in the context of a 
{\em  scale-invariant random spatial network} (SIRSN)
\cite{MR3164768,MR3082274}, studied as abstractions of road networks.  
A general conjecture is that any network built dynamically from randomly-arriving Poisson points by means of edges (now line segments in the plane)  being created to attach an arriving point to the existing network  by a ``scale-invariant rule" (that is, a rule which uses only relative distances, not absolute distances) 
should in the limit define a SIRSN.  
Of course the rule in our model ``create an edge from the newly-arrived point to the closest existing point" is about the simplest scale-invariant rule one  can imagine.\footnote{Unfortunately the tree-like structure of this model implies it does not satisfy the   
requirement of a SIRSN that mean route lengths be finite.}   
The fact that this ``simplest case" is hard to analyze suggests that  the general conjecture is very challenging.

There is extensive literature on stochastic fragmentation and coalescence models in the non-geometric ``mean-field" setting 
\cite{me-coal,bertoin}.
There is also substantial literature (see e.g. \cite{kendall_book} Chapter 9) concerning random partitions of the plane (tessellations, tilings etc).
But the combination of these themes, that is  Markovian processes of refining or coarsening partitions in the plane,
has been considered only in special refining models \cite{cowan} and in variants of the STIT model \cite{schreiber,thale}.  
The coalescing partitions process in  Theorem \ref{T:process}  is perhaps the only known self-similar Markovian process of pairwise merging partitions of $\Reals^2$ with explicit rates.
See section \ref{sec:other_coal} for further brief comments.

\section{A bound on ancestor displacement}
\label{sec:part1}
\subsection{Compactness for the marked point process}
Our first objective is to obtain a concrete bound, Proposition \ref{P1}, on the distance between the position $z_\xi$ of a particle $\xi$ (present at time $0$) and the 
position $z_{\ancestor(-t,\xi)}$ of its ancestor at time $-t$.

Some notation:
\begin{itemize}
\item $\origin$ is the origin in $\Reals^2$.
\item $|| x - y ||$ denotes Euclidean distance in $\Reals^2$.
\item $\disc(z,r)$ is the closed disc with center $z$ and radius $r$.
\item For  measurable $B \subset \Reals^2$ write $\area(B)$ for its area ($2$-dimensional Lebesgue measure) and 
$\diam(B) = \sup_{x,y \in B}  || x - y ||$ for its diameter.
\end{itemize}
\begin{Proposition}
\label{P1}
There exists a function $G(r) \downarrow  0$ as $r \uparrow  \infty$ such that,
for all $z \in \Reals^2$ and all $t > 0$, conditional on $\Xi_{\le 0}$ having a particle $\xi$ with  $z_\xi = z$ and $t_\xi > -t$ we have 
\[ G_t(r) := 
\Pr (||
z_{\ancestor(-t,\xi)} - z_\xi|| > r e^{t/2} ) \le G(r) , \ 0 < r < \infty . \]
Moreover $\int_0^\infty r G(r) dr < \infty$.
\end{Proposition}
The rest of section \ref{sec:part1} is devoted to the proof of Proposition \ref{P1} and a variant (Proposition \ref{C:ld}). 
As mentioned earlier, the conceptual point of Proposition \ref{P1} is that the distance to the time $t$ (in the past) ancestor 
is the same order of magnitude as the distance to the closest particle at that time, that is order $e^{t/2}$. 
An expression for $G(r)$ is given at (\ref{def:Gr}).

The elementary ``thinning" property of Poisson processes leads to a corresponding 
property of our space-time Poisson point process $\bm{\Xi}$.  
As $t$ runs backwards over $\infty > t > -\infty$, the processes $\Xi_{\le t}$ evolve according to the rule
\begin{quote}
each particle is deleted at stochastic rate $1$.
\end{quote}
This {\em Poisson thinning process} representation is the foundation for much of our analysis, as are the related 
{\em self-similarity} properties of our derived processes, discussed in section \ref{sec:Nota}.

To be pedantic, in forwards time we work with the filtration
$\FF_t = \sigma(\Xi_{\le t})$. 
In reversed time we work with the filtration 
\begin{equation}
\stackrel{\leftarrow} {\FF_{t}}= \sigma( (\max(t_\xi,t), z_\xi) \ : \ \xi \in \bm{\Xi} ) .
\label{f-back}
\end{equation}
So $\stackrel{\leftarrow} {\FF_{t}}$ tells us the positions of all particles, and the arrival times of particles born after time $t$,
and the following ``thinning process" property holds.
\begin{Lemma}
\label{L:rev}
Conditional on $\stackrel{\leftarrow} {\FF_{t}}$, the 
 previous lifetimes
$\{t - t_\xi : \ z_\xi \in \ZZ_{\le t} \}$ of the particles alive at time $t$ are i.i.d. with Exponential(1) distribution.
\end{Lemma}

\subsection{Derivation of an EA process}
\label{sec:line}
We study lines of descent in the genealogical tree process.
Consider a particle $\xi$ present at time $0$ at position $z_0$.  
From the  thinning process representation, its arrival time $T_0 < 0$ is such that $- T_0$ has  Exponential(1) distribution.
For $i \ge 1$ write $(T_i,Z_i)$ for the arrival time and position of its $i$'th generation ancestor, that is  $\parent[i,\xi]$.
We will show how to  represent this process 
in terms of a certain Markov process we will call the {\em excluded area} (EA) process. 

Conditional on $\{T_0 = t_0\}$ the particles present at $t_0$ are distributed as the PPP $\Xi_{< t_0}$, 
and so $\parent[i,\xi]$ is the closest such point to $z_0 $, at position $Z_1$ say.
Conditional also on $\{Z_1 = z_1\}$, 
we know there are no points of $\Xi_{< t_0}$ in the interior of
$C_1:= \disc(z_0, ||z_1 - z_0||)$.
The arrival time $T_1$ of $Z_1$ has density function 
$\propto e^t$ on $-\infty < t < t_0$, implying that 
$t_0 - T_1$ has Exponential($1$) distribution. 

Now given $T_0 = t_0$ and $(T_1,Z_1) = (t_1,z_1)$, the information we have about $\Xi_{< t_1}$ is precisely 
the fact that it has no points in $C_1$.
So $Z_2$ is the closest point to $z_1$ in a PPP of rate $e^{t_1}$ 
on $\Reals^2 \setminus C_1$.
And as before, $t_1 - T_2$ has Exponential($1$) distribution.

Now given $T_0 = t_0$ and  $(T_1,Z_1) = (t_1,z_1)$ and $(T_2,Z_2) = (t_2,z_2)$, 
we have built an ``excluded region" 
$C_2 := C_1 \cup \disc(z_1, ||z_2 - z_1||)$.
The information we have about $\Xi_{< t_2}$  is precisely 
that it is a PPP of rate $e^{t_2}$ with no points in $C_2$, 
and we can continue inductively to describe the entire process 
$((T_i, Z_i), i \ge 0)$.

\subsection{Definition of the EA process}
Here we re-specify the process above in intrinsic terms.
 Working with time $ \downarrow - \infty$ is rather 
counter-intuitive, so in the definition below it seems helpful to reverse the direction of time.

Consider the space $\Complex$ of triples $\bc = (C,z,\tau)$ such that
\[ \mbox{$C$ is a compact set in $\Reals^2$; \quad 
$z \in C$; \quad 
$0 \le  \tau < \infty$}. \]
Given an element $\bc = (C,z,\tau) \in \Complex$
we can define a probability distribution $\mu_{\bc}$
on $\Complex$ as follows.
Take a PPP $\widetilde{\Xi}$ of rate $e^{- \tau}$ on $\Reals^2 \setminus C$.
Let $\xi$ be the point of $\widetilde{\Xi}$ closest to $z$.
Set
\[ z^\prime = \xi; \ C^\prime = C \cup \disc(z,||\xi - z||); \ 
\tau^\prime = \tau + \theta\]
 where $\theta$ has Exponential(1) 
distribution independent of $\widetilde{\Xi}$.
Then let $\mu_{\bc}$ be the distribution 
of $(C^\prime, z^\prime, \tau^\prime)$.

\begin{figure}
\caption{Illustration of the standard EA process. 
$C_i$ is the union of the discs centered at $\origin, z_1,\ldots,z_{i-1}$, 
and $z_i$ is on the boundary of $C_i$.}
\label{Fig:Ci}

\setlength{\unitlength}{0.05in}
\begin{picture}(50,55)(-35,-25)
\put(0,0){\circle*{1}}
\put(0,0){\circle{10}}
\put(-3,4){\circle*{1}}
\put(-3,4){\circle{15}}
\put(1.5,10){\circle*{1}}
\put(1.5,10){\circle{25}}
\put(11.5,2.5){\circle*{1}}
\put(11.5,2.5){\circle{47}}
\put(18,25){\circle*{1}}
\put(-2,-0.5){$\origin$}
\put(-5,5){$z_1$}
\put(2,11){$z_2$}
\put(12,0.9){$z_3$}
\put(18,26.1){$z_4$}
\put(24,17){$C_4$}
\put(5,17){$C_3$}
\end{picture}
\end{figure}
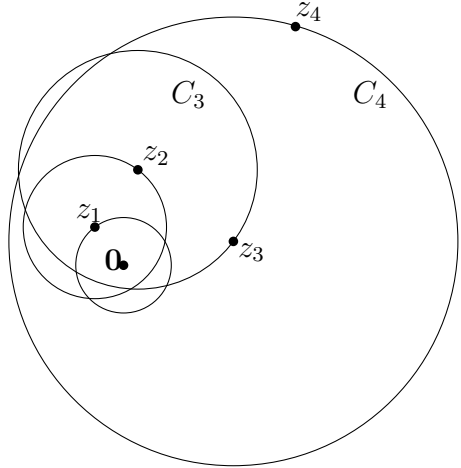

Define the {\em EA process}
to be the $\Complex$-valued Markov chain 
$(\bC_i = (C_i,Z_i,\tau_i), 0 \le i < \infty)$ 
where, for each step $i$, the conditional distribution of 
$\bC_{i+1}$ given $\bC_i= \bc$
is the distribution $\mu_{\bc}$ specified above.  
Figure \ref{Fig:Ci} provides an illustration.
It is straightforward to formalize the argument in section \ref{sec:line} to show
\begin{Lemma}
\label{L:line_descent}
Condition on $\bm{\Xi}$ containing a particle $\xi$ with $t_\xi \le 0$ and $z_\xi = z_0$.  
The process $((t_{\parent[i,\xi]}, z_{\parent[i,\xi]}), 0 \le i < \infty)$ of arrival times and positions of the ancestors of this $\xi$  
is distributed as the random process 
$((-\tau_i, Z_i),  0 \le i < \infty)$ within
the EA process $((C_i,Z_i,\tau_i), 0 \le i < \infty)$ with initial state 
$(\{z_0\}, z_0, \tau_0)$, 
where $\tau_0$ has Exponential(1) distribution.
\end{Lemma}

{\bf Terminology.}  
In what follows we write {\em step} for the steps $i$ of the EA chain, 
and {\em time} for the $\tau$'s.

\subsection{Geometric analysis of the EA process}
It is enough to study the {\em standard} EA process with
initial state
\[ (C_0,z_0,\tau_0) = (\{\origin\}, \origin, \tau_0) \] 
where $\tau_0$ has Exponential(1) distribution.\footnote{This is notationally more convenient than taking $\tau_0 = 0$, because of our convention that particles are labeled by position and arrival time.}
So in the context of Lemma \ref{L:line_descent} we will study ancestors of a particle present at position $\origin$ at time $0$.
The starting observation, Lemma \ref{LA1} below, is an expression for the growth of the area of $C_i$ at each step. 
After that we use geometric arguments to bound the diameter of $C_i$ in terms of its area.
Because $Z_i$ is on the boundary of $C_i$ this will be enough to prove Proposition \ref{P1}.

\begin{Lemma}
\label{LA1}
Conditional on $\bC_i = (C_i,z_i,\tau_i)$, the increment
$\area(C_{i+1}) - \area(C_i)$ has Exponential($e^{-\tau_i}$) distribution, independent of $\tau_{i+1} - \tau_i$.
\end{Lemma}
\begin{proof}
Writing $a_r$ for the area of 
$\disc(z_i,r) \setminus C_i$, 
\begin{eqnarray*}
\lefteqn{
\Pr (\area(C_{i+1}) - \area(C_i) > a_r  )
}
\\
&=&
\Pr(\mbox{ no point of a rate $e^{-\tau_i}$ Poisson process in 
$\disc(z_i,r) \setminus C_i$} )\\
&=& \exp(-  e^{-\tau_i} a_r) .
\end{eqnarray*}
The independence holds by construction.
\end{proof}

We can {\em lower} bound the diameter in terms of the area via
 the classical 
fact (called Bieberbach's inequality or the isodiametric inequality -- 
see \cite{MR872858} for a short proof) that the disc is extremal: 
\begin{equation}
\area(C) \le \sfrac{\pi}{4} (\diam(C))^2, \mbox{ all compact } C \subset \Reals^2 .
\label{iso}
\end{equation}
We want a corresponding {\em upper} bound, to verify that $C_j$ does not
become long and thin.  
The bound will rely upon the following geometry lemma.

\begin{Lemma}
\label{LC1}
Let $C$ be a compact set in $\Reals^2$ and let $D$ be a closed disc 
whose center is in $C$.  Then
\[
\diam(C \cup D) \le  \max \left(  \diam(C)  +
\sqrt{\sfrac{2 (\area(C \cup D) - \area(C))}{\pi}} , 
\sqrt{
 \sfrac{4 \area (C \cup D) )}{\pi} } \right) .\]
\end{Lemma}

\begin{proof}
The right side clearly bounds the distance between two points in $C$, and also between two points in $D$ because
\[ \sup_{z, z^\prime \in D} ||z - z^\prime|| = \diam(D) 
= \sqrt{\sfrac{4}{\pi}  \area (D)} 
\le \sqrt{\sfrac{4}{\pi}  \area (C \cup D)} .
\]
So it will suffice to prove the bound for one point in $C$ and the other in $D$, that is to prove
\begin{equation}
 \sup_{ z \in C, z^\prime \in D} ||z - z^\prime|| 
- \diam(C) \le \sqrt{\frac{2 (\area(C \cup D) - \area(C))}{\pi}} .
\label{zzp2}
\end{equation}
Figure \ref{Fig:CcupD} illustrates the argument.

\begin{figure}
\caption{Illustration of proof of Lemma \ref{LC1}.}
\label{Fig:CcupD}
\setlength{\unitlength}{0.06in}
\begin{picture}(50,45)(-35,-22)
\put(0,0){\circle{30}}
\put(0,0){\circle*{1.0}}
\put(4,4){\circle*{0.5}}
\put(4,4){\line(-1,1){15}}
\put(4,4){\line(1,-1){15}}
\put(10.6,10.6){\circle*{0.5}}
\put(-4,0){\vector(1,0){3.6}}
\put(-4,0){\vector(-1,0){10.6}}
\put(7.5,7.5){\vector(1,1){2.8}}
\put(7.5,7.5){\vector(-1,-1){2.9}}
\put(-0.5,-2.3){$v$}
\put(3.5,2.2){$w$}
\put(18.5,-12.8){$\ell$}
\put(10.8,11.5){$y$}
\put(7.5,6.0){$r_0$}
\put(-7,-1.8){$r$}
\qbezier(4,4)(-8,16)(-27, 2)
\qbezier(4,4)(12,-4)(4,-22)
\put(-21,3){$C$}
\put(-1,-12){$D$}
\put(19,-7){$H$}
\end{picture}
\end{figure}
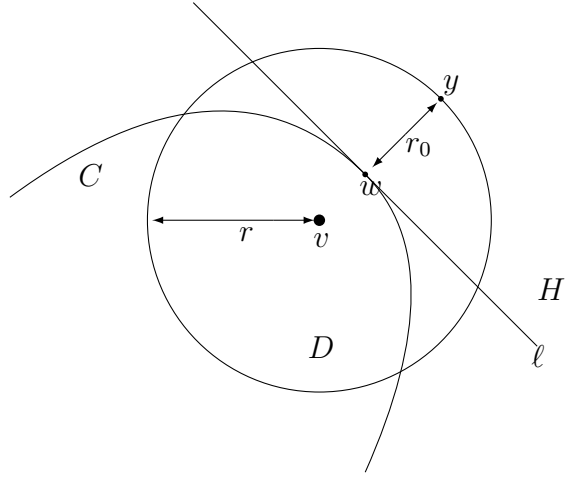

First assume $C$ is convex.  
If $D \subseteq C$ the result is trivial, so suppose not.  
Let $y$ be a point on the boundary of $D$ at maximal distance 
($= r_0$, say) from $C$, and let $w$ be a point in $C$ with 
$||y - w|| = r_0$.  
Then 
\begin{equation}
 \sup \{ ||z - z^\prime|| \ : \ z \in C, z^\prime \in D\} 
\le \diam(C) + r_0 
\label{zzp}
\end{equation}
by applying the triangle inequality to the point in $C$ closest to $z^\prime$.
Now consider the half-spaces defined by the line $\ell$ through $w$ that 
is orthogonal to the line segment $\overline{wy}$.  
The convex set $C$ must lie in the half-space not containing $y$, 
else by convexity some point in $C$ would be closer to $y$. 
And the tangent line to the disc at $y$ must be parallel to $\ell$, otherwise
some other point on the boundary would be farther from $C$ than is $y$. 
But this implies that the line segment $\overline{wy}$ is part of the 
line segment $\overline{vy}$, where $v$ is the center of the disc $D$.
So $r_0 \le r := $ radius of $D$, and 
\[  \area(C \cup D) - \area(C) \ge \area(D \cap H) \]
where $H$ is the half-space containing $y$.
Now $\area(D \cap H) $ is a certain function of $r_0$ and $r \ge r_0$, 
and clearly this function is, for fixed $r_0$, minimized at $r = r_0$, and there its value is $\sfrac{1}{2} \pi r_0^2$. So 
\[  \area(C \cup D) - \area(C) \ge  \sfrac{1}{2} \pi r_0^2\]
and combining with (\ref{zzp}) gives (\ref{zzp2}).

In proving (\ref{zzp2}) we assumed $C$ was convex.  For general $C$ we can apply 
(\ref{zzp2})  to its convex hull $C^*$ and then, noting
\[ \diam(C^*) = \diam(C), \quad 
\area(C^* \cup D) - \area(C^*) \le \area(C \cup D) - \area(C)\]
we see that (\ref{zzp2}) remains true for non-convex $C$.
\end{proof}

\subsection{Completing the proof of Proposition \ref{P1}}
\label{sec:completingP1}
Returning to the standard EA process
$\bC_i = (C_i, Z_i, \tau_i)$, we now have sufficient tools to study  $\tau_i$ and 
\[ A_i := \area(C_i), \quad D_i := \diam(C_i) . \]
From Lemma \ref{LA1} we obtain a constructive representation 
of the distribution of $((A_i,\tau_i), 0 \le i < \infty)$, as follows.  
\begin{eqnarray}
\mbox{
The process $(\tau_i, \ i \ge 0)$ forms a Poisson process of rate $1$ on 
$(0,\infty)$.} \label{tauA1}
\\ 
\mbox{ $A_i = \sum_{j=0}^{i-1} e^{\tau_j} \theta_j$ 
where $(\theta_j , j \ge 0)$ are i.i.d. Exponential($1$),}\nonumber  \\
\mbox{ independent of } (\tau_i, \ i \ge 0).
\label{area-exact}
\end{eqnarray}
Then from Lemma \ref{LC1} we get the inequality
\begin{equation}
D_{i+1} \le \max \left(
D_i + \sqrt{ \sfrac{2(A_{i+1} - A_i)}{\pi} } ,
\sqrt{ \sfrac{ 4 A_{i+1}}{\pi} } \right) .
\label{DDAA}
\end{equation}
In this section we use only the weaker inequality
\begin{equation}
 D_{i+1} \le D_i + \sqrt{ \sfrac{ 4 A_{i+1}}{\pi} } . 
\label{weaker}
\end{equation}
Because $D_0 = 0$ this implies 
\begin{equation}
 D_k \le 2\pi^{-1} \sum_{i=1}^k A_i^{1/2} . 
 \label{Dk2}
 \end{equation}
Because
\[ A_i^{1/2} = \left(\sum_{j=0}^{i-1} e^{\tau_j} \theta_j\right)^{1/2} 
\le \sum_{j=0}^{i-1}  (e^{\tau_j} \theta_j)^{1/2} \]
we find that
\begin{equation}
D_k \leq \overline{D}_k:= 2\pi^{-1/2} \sum_{j=0}^{k-1} 
(k-j)  e^{\tau_{j}/2} \theta_{j}^{1/2} .
\label{Dkk}
\end{equation}
In Proposition \ref{P1} we seek to bound the probability of the event
$ \{ ||z_{\ancestor(-t,\xi)} - \origin ||  > r e^{t/2} \}$
for a particle $\xi$ at time $0$ with position $z_\xi = \origin$ 
(the case of general $z_\xi = z$ is the same, by translation-invariance). 
Fix $t$.
 Identifying the EA process with the ``line of descent" process as in Lemma \ref{L:line_descent}, the position 
 $z_{\ancestor(-t,\xi)}$ is by construction on the boundary of the region 
 $C_{N(t) + 1}$ for 
 \[ N(t) = \max \{i: \tau_i < t \} . \] 
 Therefore, using (\ref{Dkk}),
 \begin{equation}
  ||z_{\ancestor(-t,\xi)} - \origin || \le D_{N(t) + 1} \le 2\pi^{-1/2} \sum_{j=0}^{N(t)} 
(N(t)+1-j)  e^{\tau_{j}/2} \theta_{j}^{1/2} . 
\label{zat}
\end{equation}
From properties of the rate-$1$ Poisson process $(\tau_j, j \ge 0)$ on $(0, \infty)$, the time-points
$(t - \tau_{N(t)},   t - \tau_{N(t)-1},  t - \tau_{N(t)-1}, \ldots t -  \tau_0)$ 
are distributed as an initial segment of a rate-$1$ Poisson process $(\sigma_j, j \ge 1)$ on $(0, \infty)$.
So rewriting (\ref{zat}) in terms of the $(\sigma_j)$ and $u = N(t) + 1 - j$ gives 
\begin{equation}
e^{-t/2}   ||z_{\ancestor(-t,\xi)} - \origin || 
\mbox{ is stochastically dominated by } \chi :=   2\pi^{-1/2} \sum_{u=1}^\infty  u e^{- \sigma_u} \eta_u^{1/2} 
\label{etz}
\end{equation}
where  $(\eta_u, u \ge 1)$ are i.i.d. Exponential($1$), independent of  the Poisson process  $(\sigma_u, u \ge 1)$.
So  Proposition \ref{P1} holds for 
\begin{equation}
G(r):= \Pr(\chi > r), \ 0 < r < \infty 
\label{def:Gr}
\end{equation}
and it is easy to check that $\int_0^\infty r G(r) dr < \infty$.

\subsection{A large deviation bound for occupation times}
\label{sec:LD}
The following technical bound will enable us to bound the distance required for two lines of descent to merge (Proposition \ref{Pcouple} later).

\begin{Proposition}
\label{C:ld}
For the standard EA process write 
\[ B  = \cup_{ \{i:   
e^{-\tau_i/2} \diam(C_{i}) > b 
\}} 
[\tau_{i-1}, \tau_{i}) . \]
Then for sufficiently large $b $ there exist $A < \infty$ and $\rho > 0 $ such that
\begin{equation}
\Pr ( \Leb (B \cap [0,T] ) > T/3) \le A \exp(-\rho T), \quad 0<T<\infty 
\label{ld1}
\end{equation}
where $\Leb$ denotes Lebesgue measure.
\end{Proposition}

We previously used inequality (\ref{weaker}) to bound $D_i := \diam(C_i)$ in terms of the areas $A_i = \area(A_i)$. 
Here we will use a slightly different bound. 
\begin{Lemma}
\label{L:DiA}
$D_i \le \sqrt{ \sfrac{ 4 A_{i}}{\pi} } + \sum_{j=0}^{i-1} \sqrt{ \sfrac{2(A_{j+1} - A_j)}{\pi} } $.
\end{Lemma}
\begin{proof}
Setting $\widetilde{D}_i := D_i - \sqrt{ \sfrac{ 4 A_{i}}{\pi} } $, inequality (\ref{DDAA}) becomes
\[ \widetilde{D}_{i+1} \le 
 \max \left( 
\widetilde{D}_i + \sqrt{ \sfrac{2(A_{i+1} - A_i)}{\pi} } 
+ \left(  \sqrt{ \sfrac{ 4 A_{i}}{\pi} } - 
\sqrt{ \sfrac{ 4 A_{i+1}}{\pi} }\right) , 0
 \right) . \]
But the term $ (  \sqrt{ \sfrac{ 4 A_{i}}{\pi} } - 
\sqrt{ \sfrac{ 4 A_{i+1}}{\pi} }) $ is negative, and $\widetilde{D}_0 = 0$, so we find
\[ \widetilde{D}_i \le \sum_{j=0}^{i-1} \sqrt{ \sfrac{2(A_{j+1} - A_j)}{\pi} } \]
establishing the asserted bound.
\end{proof}

Recall the notation from (\ref{tauA1},\ref{area-exact}): 
 $(\tau_j, \ j \ge 0)$ denotes  a Poisson process of rate $1$ on 
$(0,\infty)$, and  $(\theta_j , j \ge 0)$ denotes  i.i.d. Exponential($1$) random variables
independent of $(\tau_i, \ i \ge 0)$. 
From (\ref{area-exact}) and Lemma \ref{L:DiA}, 
\begin{equation}
 D_i  \le \sqrt{ \sfrac{ 4 A_{i}}{\pi} } + \sqrt{ \sfrac{2}{\pi}} \  D^*_i 
 \label{DiA}
 \end{equation}
where
\[ A_i = \sum_{j=0}^{i-1} e^{\tau_j} \theta_j; \quad 
D^*_i  =  \sum_{j=0}^{i-1} e^{\tau_j /2} \theta_j^{1/2} . \]
To prove Proposition \ref{C:ld} we will rephrase  inequalities from
section \ref{sec:completingP1} in terms of continuous-time processes. 
Set $V_0 = 0$ and define $(V_t, 0 \le t < \infty)$ to be the process which increments by 
$\theta_j$ at time $\tau_j$, and otherwise decreases at exponential rate $1$.  
In symbols, writing $N^j_t = \indic_{\{ t \ge \tau_j\}}$, 
\[
dV_t = - V_t dt + \sum_{j \ge 0} \theta_j dN^j_t 
. \]
At time $\tau_i -$ (just before the jump at $\tau_i$) we have
\[ V_{\tau_i -} = \sum_{j = 0}^{i-1} \theta_j e^{\tau_j - \tau_i} 
= e^{-\tau_i} A_i . \] 
Because $V_t$ is decreasing on $[\tau_{i-1},\tau_i)$ we have 
\begin{equation}
 \mbox{ if } e^{-\tau_i}A_i > b_A \mbox{ then } V_t > b_A \mbox{ on } [\tau_{i-1},\tau_i) . 
 \label{VA1}
 \end{equation}
Similarly, define $(W_t, 0 \le t < \infty)$ to be the process which increments by 
$\theta^{1/2} _j$ at time $\tau_j$, and otherwise decreases at exponential rate $1/2$.  
Take $W_0 = 0$.
In symbols, 
\[
dW_t = -  \sfrac{1}{2} W_t dt + \sum_j \theta^{1/2}_j dN^j_t 
. \]
At time $\tau_i -$ we have
\[ W_{\tau_i -} = \sum_{j = 0}^{i-1} \theta^{1/2}_j e^{(\tau_j - \tau_i)/2} 
= e^{-\tau_i/2} D^*_i . \] 
So as at (\ref{VA1})
\begin{equation}
 \mbox{ if } e^{-\tau_i/2}D^*_i > b_D \mbox{ then } W_t > b_D \mbox{ on } [\tau_{i-1},\tau_i) . 
 \label{VA2}
 \end{equation}
Combining (\ref{DiA}) with (\ref{VA1}, \ref{VA2}) for appropriate $b_A, b_D$ defined in terms of $b$, we now see 
that the proof of  Proposition \ref{C:ld} reduces to proofs of large deviation bounds for occupation measures 
of the processes $(V_t)$ and $(W_t)$.  
That is, it suffices to prove 
\begin{Proposition}
For sufficiently large $b $ there exist $A < \infty$ and $\rho > 0 $ such that
\begin{eqnarray}
\Pr \left( \int_0^T \indic_{ \{V_t > b \}} \ dt  > T/6 \right) \le A \exp(-\rho T), \quad 0<T<\infty 
\label{ld2} \\
\Pr \left( \int_0^T \indic_{ \{W_t >  b \}} \ dt  > T/6 \right) \le A \exp(-\rho T), \quad 0<T<\infty 
\label{ld3}.
\end{eqnarray}
\end{Proposition}
We will give the proof for $(V_t)$, and then note that essentially the same proof works for $(W_t)$.

Fix a high level $b$.
The process regenerates at each downcrossing of $b$. 
So starting from the first downcrossing there is an
 i.i.d. sequence $((L_b(i), K_b(i)), i \ge 1)$ where $L_b$ is the duration  and $K_b$ is the ``occupation time above $b$" between successive downcrossings. 
We can decompose $L_b$ as $L^\prime_b + K_b$ where 
$L^\prime_b$ is the time until first upcrossing of $b$ and 
$K_b$ is the subsequent time until the next downcrossing of $b$.
It is easy to see
\begin{equation}
L^\prime_b \to_p \infty \mbox{ as } b \to \infty .
\label{Lb1}
\end{equation}
It is also easy to see that $K_b \to_p 0$ as $b \to \infty$, 
though we need the stronger result  
\begin{equation}
\lim_{b \to \infty} \Ex \exp(\theta K_b) = 1, \quad 0 < \theta < \infty.
\label{Lb2}
\end{equation}
To prove this, note that during an excursion above $b$ the process $(V_t)$ is upper bounded by the process $(V^*_t)$ in
which the drift term is $-b \ dt$ instead of $- V_t \ dt$.  But the process $(V^*_t)$ describes the workload in a $M/M/1$ queue with arrival rate $1$ 
and service rate $b$.  
So the distribution of $K_b$ is stochastically smaller than the server's busy period in that queue, and from classical exact formulas for that busy period distribution
(e.g.  \cite{asmussen}) one can deduce (\ref{Lb2}).

Writing $\tau_n$ for the time of the $n$'th regeneration, (\ref{Lb2}) 
and the classical large deviation theorem for i.i.d. sums  imply that for $b$ sufficiently large
\[
\Pr \left( \int_0^{\tau_n} \indic_{ \{V_t > b \}} \ dt  > n/6 \right) 
= \Pr \left(\sum_{i=1}^n K_b(i) > n/6 \right) 
\mbox{ decreases exponentially } 
\]
as $n \to \infty$.
Another use of the i.i.d. large deviation theorem applied to the time intervals $L^\prime_b$  at (\ref{Lb1}) implies that for $b$ sufficiently large the probabilities 
$\Pr( \tau_n < n)$  decrease exponentially as $n \to \infty$, and this establishes (\ref{ld2}).

The argument for (\ref{ld3}) is essentially the same, and this
completes the proof of Proposition  \ref{C:ld}.

\section{Coalescence of lines of descent}
\label{sec:joint}
In this section we continue the style of analysis in section \ref{sec:part1} by 
studying the lines of descent of two particles present at time $0$.
This involves a coupled EA process, whose dynamics are described in section \ref{sec:coupled}. 
Note that Proposition \ref{P1} implies that, for particles at distance $r  \gg 1$ apart, 
the ``coalesce" time (time backwards to their most recent common ancestor) must be at least 
$(2 - o(1)) \log r$.
Our goal is to give an upper bound, Proposition \ref{Pcouple}, on the coalesce time distribution.
The central idea is to 
 use Proposition \ref{C:ld} to show that, if not coalesced already, the lines of descent at time $-t$ are typically only order $e^{t/2}$ apart 
(the same order as the distance to the nearest time $-t$ particle): this is Proposition \ref{P:diamcoup}. 
A geometric argument then shows (Lemma \ref{L:coal})  
that there is a non-vanishing chance to merge in the next generation backwards. 
These ingredients are combined in section \ref{sec:ppcouple} to prove Proposition \ref{Pcouple}.

\subsection{The coupled EA process}
\label{sec:coupled}
Fix $t_0 \ge 0$.
For the rest of section \ref{sec:joint} we condition on the time-$t_0$ configuration $\ZZ_{\le t_0}$ 
containing a particle at position $z^1_0$ and  the time-$0$ configuration $\ZZ_{\le 0}$ 
containing a particle at position $z^2_0$.
The distribution of the line of descent for each particle is just a translated and scaled version of
the distribution of  the EA process in Lemma \ref{L:line_descent}.
So we anticipate that the joint distribution of the two lines of descent can be described in terms of some suitably coupled 
versions of the EA process.

Precisely, we will specify the {\em coupled EA process} 
\[
((\bC^1_i; \bC^2_i),   \ i = 0,1,2,\ldots )   = ((C^1_i,z^1_i,\tau^1_i); (C^2_i,z^2_i,\tau^2_i)), \ i = 0,1,2,\ldots )
\]
with initial states $(\{z^1_0\}, z^1_0,\tau^1_0)$  and 
 $(\{z^2_0\}, z^2_0, \tau^2_0)$ 
where $\tau^2_0$ and $\tau^1_0 + t_0$ are independent Exponential(1).   
At each step (before the coalescence step $\Icouple$ below) only one of the components 
($\bC^1_i$ or $\bC^2_i$) changes. 
There are notational issues in describing this coupled processes.  
We write
$(\bC^1_i; \bC^2_i)$ for the configuration after the $i$'th step of the coupled process. 
Because only one component changes in each step before coalescence, we need different notation for the configuration of a given component after $j$ changes 
of that particular component, and we write
$(\bC^1_{(j)})$ and $(\bC^2_{(j)})$ for these ``jump processes" of each component. 
And it is these jump processes  which individually are evolving as the ordinary EA process.

The evolution rule for the coupled process, which will be derived from
the dynamics of the underlying tree process as was done in Lemma \ref{L:line_descent},
is as follows.

\begin{quote}
 Write the configuration after $i$ steps as $(\bC^1_i; \bC^2_i)$. 
Before step $\Icouple$ we must have $\tau_i^1 \ne \tau^2_i$; 
 suppose $\tau_i^1 < \tau^2_i$ (the other case is symmetric).
Take a PPP of rate $e^{-\tau^1_i}$ on 
$\Reals^2 \setminus (C^1_i \cup C^2_i)$, but augmented with an extra point planted
at $z^2_i$.  
Let $\xi$ be the point of the augmented PPP closest to $z^1_i$.  
If $\xi = z^2_i$ then we say that the process coalesces at
position $z^2_i$ and time $\tau^2_i$; 
write $\Icouple = i+1$ for the coalesce step, 
$\Zcouple = z^2_i$ for the coalesce position,
$\Tcouple = \tau^2_i$ for the coalesce time.
Otherwise set $C^1_{i+1} = C^1_i \cup \disc(z^1_i, ||\xi - z^1_i||)$, 
set $z^1_{i+1} = \xi$ and take 
$\tau^1_{i+1} = \tau^1_i + \theta$
 where $\theta$ has Exponential(1) 
distribution independent of all previously constructed random variables.  
Set $ \bC^2_{i+1} =  \bC^2_i$.
\end{quote}

\noindent
Note in particular that the  configuration after $i$ steps determines the value
\begin{equation}
\tau_i := \min(\tau^1_i, \tau^2_i);
\end{equation}
 the {\em arg min} determines which component will change on the next step
and $\tau_i$ determines the rate $e^{- \tau_i}$  of the PPP used to construct the next step.

{\bf Remark.} 
For completeness, let us give the behavior of the coupled process after coalescence, though this is not directly relevant to our arguments. 
If the coalesce step is
 $i+1$ as above, then $z^1_{i+1} = z^2_{i+1}$ and 
$\tau^1_{i+1} = \tau^2_{i+1}$ but maybe $C^1_{i+1} \neq C^2_{i+1}$.  In  subsequent steps $k$ we use the same PPP outside 
$C^1_k \cup C^2_k$ for each component, and therefore we have 
$(\tau^1_k, z^1_k) = (\tau^2_k,z^2_k)$ for all $k \ge i+1$.
Each of the two component jump processes does evolve as the EA process, except that for the first component process there is the extra planted point at $z^2_0$.  But this extra point
 only comes into play if it is the exact point of coalescence, 
and so does not affect our arguments for upper bounding the 
coalesce time.

\medskip \noindent
A realization of six initial steps of the coupled process is illustrated in Figure \ref{Fig:coupled1}.
On the left are the successive positions 
$z^1_0, z^1_{(1)}, z^1_{(2)}, z^1_{(3)}, z^1_{(4)}$ 
in the first component process, and on the right are the positions 
in the second component process. 
The associated times for the first process are 
$ -t_0 < \tau^1_0 =  \tau^1_{(0)} < \tau^1_{(1)} < \tau^1_{(2)} <  \tau^1_{(3)} <  
 \tau^1_{(4)} $ 
and for the second process are 
$0 < \tau^2_0 = \tau^2_{(0)} < \tau^2_{(1)} < \tau^2_{(2)} $.
Suppose that the times associated with these steps are ordered as
\[ -t_0 < \tau^1_{(0)} < \tau^1_{(1)} < \tau^1_{(2)} < \tau^2_{(0)} < \tau^1_{(3)} < \tau^2_{(1)} . \]
In terms of steps $i$ of the coupled process (indicated in the figure as $a, b, c, d, e, f$)
we have

\medskip
\begin{center}
$\begin{array}{cccccccc}
i & 0 & a1 & b2 & c3 & d4 & e5 & f6 \\
z^1_i = & z^1_{0} & z^1_{(1)} & z^1_{(2)}   & z^1_{(3)}  & z^1_{(3)}  & z^1_{(4)}  & z^1_{(4)} \\
z^2_i = & z^2_0 & z^2_0 & z^2_0 & z^2_{0} & z^2_{(1)} & z^2_{(1)}  & z^2_{(2)} \\
\tau^1_i = & \tau^1_0 & \tau^1_{(1)} & \tau^1_{(2)}   & \tau^1_{(3)}  & \tau^1_{(3)}  & \tau^1_{(4)}  & \tau^1_{(4)} \\
\tau^2_i = & \tau^2_0 & \tau^2_0 & \tau^2_0 & \tau^2_{0} & \tau^2_{(1)} & \tau^2_{(1)} & \tau^2_{(2)} \\
\tau_i = & \tau^1_0 & \tau^1_{(1)} & \tau^1_{(2)}   & \tau^1_{0} & \tau^1_{(3)}  & \tau^1_{(1)} & ??
\end{array}$
\end{center}

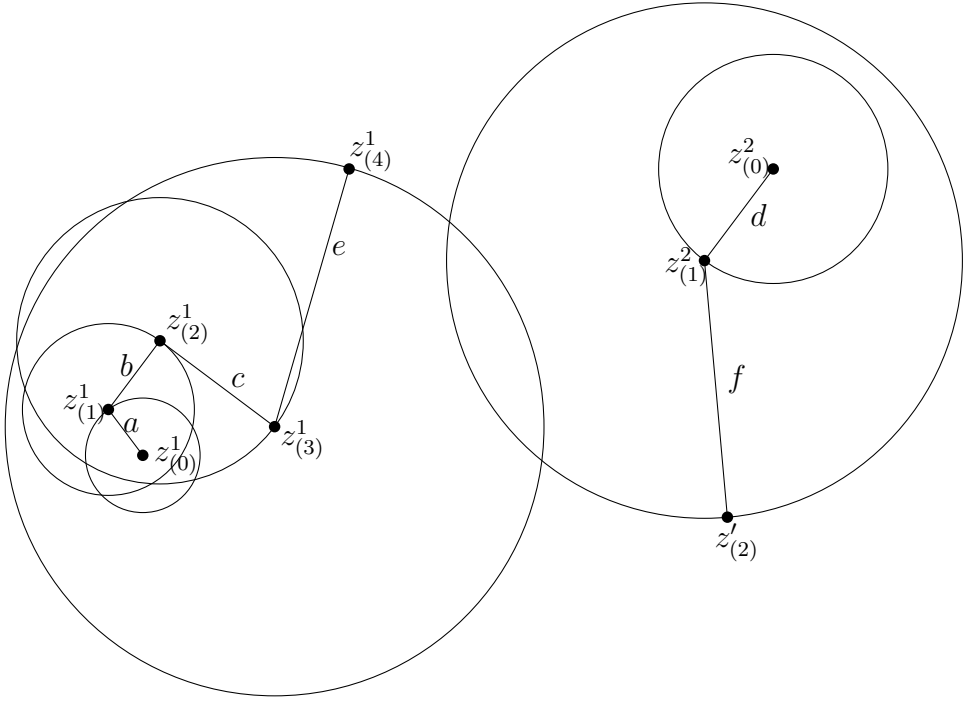
\begin{figure}
\caption{Initial steps of the coupled EA process.}
\label{Fig:coupled1}
\setlength{\unitlength}{0.06in}
\begin{picture}(50,68)(-10,-25)
\put(0,0){\circle*{1}}
\put(0,0){\circle{10}}
\put(-3,4){\circle*{1}}
\put(-3,4){\circle{15}}
\put(1.5,10){\circle*{1}}
\put(1.5,10){\circle{25}}
\put(11.5,2.5){\circle*{1}}
\put(11.5,2.5){\circle{47}}
\put(18,25){\circle*{1}}
\put(1,-0.5){$z^1_{(0)}$}
\put(-7,4){$z^1_{(1)}$}
\put(2,11){$z^1_{(2)}$}
\put(12,0.9){$z^1_{(3)}$}
\put(18,26.1){$z^1_{(4)}$}
\put(0,0){\line(-3,4){3}}
\put(-3,4){\line(3,4){4.5}}
\put(1.5,10){\line(4,-3){10}}
\put(11.5,2.5){\line(6.5,22.5){6.5}}
\put(55,25){\circle*{1}}
\put(55,25){\circle{20}}
\put(49,17){\circle*{1}}
\put(49,17){\circle{45}}
\put(51,25.5){$z^2_{(0)}$}
\put(45.5,16){$z^2_{(1)}$}
\put(51,-5.4){\circle*{1}}
\put(49.9,-8){$z^\prime_{(2)}$}
\put(55,25){\line(-3,-4){6}}
\put(49,17){\line(2,-22.4){2}}
\put(51,6){$f$}
\put(53,20){$d$}
\put(16.5,17.5){$e$}
\put(7.7,6){$c$}
\put(-2,6.9){$b$}
\put(-1.7,2){$a$}
\end{picture}
\end{figure}

We will now relate how this description of the coupled EA process arises as the description of the two lines of descent 
within the tree process for the two given points of $\ZZ_{\le t_0}$  and  $\ZZ_{\le 0}$  at positions $z^1_0$ and $z^2_0$.
Consider Figure \ref{Fig:coupled1}.
Inductively, we have traced back the two lines of descent to 
$(-\tau^1_{(4)},z^1_{(4)})$ and $(-\tau^2_{(2)},z^2_{(2)})$, using $6$ steps of the coupled process.
What happens next depends on which of $\tau^1_{(4)}$ or $\tau^2_{(2)}$ (that is, which of $\tau^1_6$ or $\tau^2_6$) 
is smaller.  
Taking the case $\tau^1_{(4)} < \tau^1_{(2)}$ 
(the other case is symmetric) then, 
to find the parent of $(-\tau^1_{(4)},z^1_{(4)})$ in the tree process,
 we need to search the region where vertices may have arrived before 
$-\tau^1_{(4)}$; this excludes both $C^1_{(4)}$ and the interior of 
$C^2_{(2)}$, because the latter contains no particles arriving before 
$-\tau^2_1$, and we must have $\tau^2_{(1)} < \tau^1_{(4)}$ because 
of the rule that the component with smaller $\tau$-value expanded first.
However the particle at $z^2_{(2)}$ arrived at time $-\tau^2_{(2)}$, 
which was before time
$-\tau^1_{(4)}$, and so is eligible to be the parent of $(-\tau^1_{(4)},z^1_{(4)})$.
So the parent of $(-\tau^1_{(4)},z^1_{(4)})$ is the closest particle to $z^1_{(4)}$ in the 
Poisson process $\ZZ_{ \le - \tau^1_{(4)}}$, which has rate $e^{-\tau^1_{(4)}}$, on the complement 
of $C^1_{(4)} \cup C^1_{(2)}$, or is $z^1_{(2)}$ if that particle is closer. 
In the latter case the two lines of descent coalesce at 
$(-\tau^1_{(2)},z^1_{(2)})$.

In terms of steps $i$ of the coupled process, $\tau^1_{(4)} = \tau^1_6$ and
$C^1_{(4)} \cup C^2_{(2)} =  C^1_6 \cup C^2_6$ and 
$(-\tau^2_{(2)},z^2_{(2)}) = (-\tau^2_6,z^2_6)$.
This completes the derivation of the evolution rule stated 
in section \ref{sec:coupled}.

To summarize:
\begin{Lemma}
\label{L:line_descent_2}
Condition on $\bm{\Xi}$ containing  particles $\xi^1$ and $\xi^2$ with $t_{\xi^1} \le t_0, \ t_{\xi^2} \le 0$ and $(z_{\xi ^1}, z_{\xi^2}) = (z^1_0, z^2_0)$.  
The joint process 
\[((t_{\parent[i,\xi^1]}, z_{\parent[i,\xi^1]},   \  t_{\parent[i,\xi^2]}, z_{\parent[i,\xi^2]}), 0 \le i < \infty)\]
 of arrival times and positions of the ancestors of these particles 
is distributed as the random process 
$((-\tau^1_{(i)}, Z^1_{(i)},  \  -\tau^2_{(i)}, Z^2_{(i)}),  0 \le i < \infty)$ within
the coupled EA process 
$((C^1_i,Z^1_i,\tau^1_i); ( C^2_i,Z^2_i,\tau^2_i) , \   0 \le i < \infty)$ with initial states 
$(\{z^1_0\}, z^1_0, \tau^1_{(0)})$ and  $(\{z^2_0\}, z^2_0, \tau^2_{(0)})$
where  $\tau^1_{(0)} + t_0$ and  $ \tau^2_{(0)}$ are independent with Exponential(1) distribution.
\end{Lemma}

As mentioned before, we write $\Icouple$ for the first step $I$ such that 
 $z^1_I = z^2_I$ (and call that point $\Zcouple$), or equivalently 
for the first step $I$ such that 
 $\tau^1_I = \tau^2_I$ (and call that time $\Tcouple$).
 We can now state the main result of section \ref{sec:joint}.
 \begin{Proposition}
\label{Pcouple}
There exist constants $K, \beta < \infty$ and $\rho > 0$ such that,
in the coupled EA process above, for any $(z^1_0, z^2_0)$ and any $t_0 \ge 0$,
\begin{equation}
\Pr(\Tcouple > t) \le K \exp(-\rho t) \mbox{ for all }
t > \beta \log^+ || z^1_0 -  z^2_0 || .
\label{Tcoal}
\end{equation}
\end{Proposition}

 Heuristically we expect $|| \Zcouple - z^1_0|| \asymp \exp( \Tcouple /2)$ in the tails, and so tail behavior of the form
 $\Pr(\Tcouple > t) \asymp \exp(- \gamma t/2)$ 
 would be equivalent to tail behavior of the form 
$\Pr(  || \Zcouple ||  > r) \asymp r^{- \gamma}$.
 We conjecture the latter is true; precisely, that the exponent
 \begin{equation}
 \gamma := -  \lim_{r \to \infty} \frac{\log \Pr(  || \Zcouple ||  > r)}{\log r} 
 \label{gamma:conj}
 \end{equation}
 exists and does not depend on $|| z^1_0 -  z^2_0 || $ or $t_0$.
 This is closely related to the ``are boundaries fractal?" issue, 
 as will be discussed in section \ref{sec:2heuristics}.

{\em For ease of exposition we will present the proof of Proposition \ref{Pcouple} in the case $t_0 = 0$. 
The general case requires only minor modifications, noted below.}

\subsection{The coalescence step}
\label{sec:coalesce}

Here we will give conditions to ensure a non-vanishing probability of coalescing at the next step.
This requires only a simple geometric lemma.
\begin{Lemma}
\label{Lzz}
Let $z_1, z_2 \in \Reals^2$, and let $r, \lambda > 0$.  Let $\ZZ^\lambda$ be a Poisson point process of 
rate $\lambda$ on $\Reals^2$.  Then the event

the nearest point $z \in \ZZ^\lambda$ to $z_1$ is also the nearest point to $z_2$, 
and 

$\min (||z - z_1||, ||z - z_2||) > r$

\noindent
has probability at least
\begin{equation}
\exp(- \lambda \pi (r + ||z_2 - z_1||)^2) - 2 c_0 \lambda^{1/2} ||z_2 - z_1||\label{couple1}
\end{equation}
for a certain absolute constant $c_0$.
\end{Lemma}
\begin{proof}
Write $d = ||z_2 - z_1||$.  
Consider the events

(A): the distance from $z_1$ to the nearest point of $\ZZ^\lambda$ is at least 
$r+d$; 

(B): the distances $D_1^{(\lambda)}$ and  $D_2^{(\lambda)}$ 
 from $z_2$ to the nearest two points of $\ZZ^\lambda$ are 
such that
$D_2^{(\lambda)} - D_1^{(\lambda)} \le 2d$.

\medskip
\noindent
We assert that, in order that the event in Lemma \ref{Lzz} occurs, it is sufficient that the event
(A) occurs and the event (B) does not occur. 
To prove this assertion, let $z^\prime_1$ be the closest point of $\ZZ^\lambda$ to $z_1$, and  
let $z^\prime_2$ be the closest point of $\ZZ^\lambda$ to $z_2$. 
Suppose $z^\prime_2 \neq z^\prime_1$.
By the triangle inequality
\[ ||z_1 - z^\prime_2|| \le d + ||z_2 - z^\prime_2|| 
\mbox{ and } 
||z_2 - z^\prime_1|| \le d + ||z_1 - z^\prime_1|| . \]
Because $z^\prime_1$ is the closest point to $z_1$
\[ ||z_1 - z^\prime_1|| < ||z_1 - z^\prime_2|| . \]
Combining these three inequalities leads to
\[ ||z_2 - z^\prime_1|| - ||z_2 - z^\prime_2|| < 2d . \]
So if (B) fails then then $z^\prime_1 = z^\prime_2 = z$ say.
And so if (A) holds then, by the triangle inequality, 
$\min (||z - z_1||, ||z - z_2||) > r$.

The probability $\Pr(A)$  equals the first term
in (\ref{couple1}).  It is easy to check that 
$D_2^{(1)} - D_1^{(1)}$  has a density bounded by some $c_0$, so (by scaling) the density of 
$D_2^{(\lambda)} - D_1^{(\lambda)}$ is bounded by $c_0 \lambda^{1/2}$. 
So $\Pr(B)$ is at most $2 c_0 \lambda^{1/2}d$. 
\end{proof}

We now apply this to the coalescence step.
\begin{Lemma}
\label{L:coal}
Consider a state $(\bc^1,\bc^2) = 
((C^1,z^1,\tau^1), (C^2,z^2,\tau^2))$
of the coupled EA process started at $z^1_0$ and 
$z^2_0$.
Suppose $\tau^1 < \tau^2$.  
Write 
$ \bar{\Delta} = \max(\diam(C^1), \diam(C^2)$.
Then the probability that the process coalesces at the next step
is at least
\begin{equation}
\exp(- 4\pi e^{-\tau^1}  (||z^1_0 - z^2_0||+2\bar{\Delta})^2)
- 2 c_0 || z^1_0 - z^2_0 || e^{-\tau^1/2}  .
\label{pitz}
\end{equation}
\end{Lemma}
\begin{proof}
Because $z^1_0 \in C^1$ and $z^2_0 \in C^2$,  we have
\[ ||z^1 - z^2|| \le || z^1_0 - z^2_0 || + 2\bar{\Delta} := r \]
and moreover each 
$\disc(z^i,r)$ contains $C^1 \cup C^2$. 
We can now apply Lemma \ref{Lzz} with 
$\lambda = e^{-\tau^1}$; if the event in  Lemma \ref{Lzz} occurs
then the process coalesces at the next step.
\end{proof}

\subsection{Diameters in the coupled process}
\label{sec:diamcoup}
The key ingredient  in proving Proposition \ref{Pcouple} is the following extension of Proposition \ref{C:ld} to the coupled EA process, 
which will enable us to apply Lemma \ref{L:coal}.
This extension looks ``obvious" but the proof is rather fussy. 

Fix large $b$, and regard a step $i$ of the coupled EA process as ``good" if 
\begin{equation}
e^{- \tau_i/2} 
\max( \diam(C^1_i), \diam (C^2_i) ) 
\le b .
\label{tCC}
\end{equation} 
Let $N_b(T)$ be the number of ``good" steps before\footnote{In the general case $t_0>0$ we only count the number of steps with $\tau^1_i > 0$.}
time $T$.
\begin{Proposition}
\label{P:diamcoup}
In the coupled EA process, for sufficiently large $b$ there exist constants $a, \rho > 0$ and $K< \infty$ such that
\[ \Pr(\Tcouple > T, \ N_b(T) \le aT) \le K \exp(-\rho T), \ 0 < T < \infty . \]
\end{Proposition}
The bound does not depend on the distance $|| z^1_0 - z^2_0||$ between the particles.

In this section \ref{sec:diamcoup}, for an event indexed by $T$ we say the event
``has vanishing probability" if the probabilities are $O(\exp(- \rho T))$ as $T \to \infty$, for some $\rho > 0$.
As in Proposition \ref{C:ld}, for $j = 1,2$ write
\[ B^j  = \cup_{ \{i:   
e^{-\tau^j_i/2} \diam(C^j_{i}) > b 
\}} 
[\tau^j_{i-1}, \tau^j_{i}) . \] 
Note this is the same if use the indices $\tau^j_{(i)}$ of the jump processes.

Write
\[ \widetilde{N}_b(T) := \Leb ((B^1 \cup B^2)^c \cap [0,T] ) .\] 
In words, this is the duration of time for which a ``good" event similar to (\ref{tCC}) is occurring. 
Applying Proposition \ref{C:ld} to both components of the coupled process, for sufficiently large $b$ 
\begin{equation}
 \mbox{ the event } \{\Tcouple > T, \    \widetilde{N}_b(T)  < T/3   \} 
\mbox{ has vanishing probability}. 
\label{TcTv}
\end{equation}
This is almost what we are trying to prove as Proposition \ref{P:diamcoup}, 
except that we need to switch from ``duration of time" $ \widetilde{N}_b(T)$ to ``number of steps" $N_b(T)$.

In the construction of the coupled EA process we can start with two independent rate-$1$ PPPs on $(0,\infty)$ and use these as the values of 
$\tau^1_{(i)}$ and $\tau^2_{(i)}$ until the coalescence step.  
So on the event $\{\Tcouple > T\}$ these times, within $[0,T]$,  coincide with the times of  two independent rate-$1$ PPPs.
Here is a helpful way to record a consequence of this fact.\footnote{The fact is slightly subtle, in that the previous assertion is not true
{\em conditional on the event $\{\Tcouple > T\}$}.}
\begin{Lemma}
\label{L:indePoi}
Let $B_T$ be an event defined in terms of two independent rate-$1$ PPPs 
on $[0,T]$.
Let $B^*_T$ be the corresponding event 
defined in terms of the times 
$ \tau^1_{(0)} < \tau^1_{(1)} <  \tau^1_{(2)} < \ldots < T$ 
and
$ \tau^2_{(0)} < \tau^2_{(1)} <  \tau^2_{(2)} < \ldots < T$ 
in the coupled EA process.  Then
\[ \Pr( \Tcouple > T, \ B^*_T)  \le \Pr(B_T)  . \]
\end{Lemma}
We will apply this to events $B_T$ which have vanishing probability, in which setting Lemma \ref{L:indePoi} 
says we can ignore such events for the purpose of proving Proposition \ref{P:diamcoup}.
We state the following  straightforward 
large deviation bounds for quantities associated with the PPP.
\begin{Lemma}
\label{L:rate2}
Let $(\tau_i)$ be a rate-$2$ PPP on $(0,\infty)$, let 
$a > 0$ and let $H_T(a)$ be the sum of the lengths of the 
$\lfloor aT \rfloor$ longest intervals $(\tau_i, \tau_{i+1})$ with $\tau_i < T$.
Then, for sufficiently small $a$, 
\begin{equation}
 \mbox{ the event }
 \{ H_T(a) \ge T/4 \}
 \mbox{ has vanishing probability}. 
\label{HbT}
\end{equation}
\end{Lemma}
\begin{Lemma}
\label{L:rate3}
Let $(\tau_i)$ be a rate-$2$ PPP on $(0,\infty)$, represented as the superposition of two independent rate-$1$ PPPs.
Define $\tau_k^+$ as the minimum value of $\tau_\ell$ for $\ell \ge k+2$ such that the events at $\tau_{k+1}. \tau_{k+2}, \ldots, \tau_{\ell}$ include events 
from both component processes.  Define
\[ G_T(a^\prime) = \sum \{  \tau_k^+ - \tau_k \ : \  \tau_k \le T, \ \tau_k^+ - \tau_k > a^\prime \} . \]
Then, for sufficiently large $a^\prime$, 
\begin{equation}
 \mbox{ the event }
 \{ G_T(a^\prime) \ge T/12 \}
  \mbox{ has vanishing probability}.
\label{HbT2}
\end{equation}
\end{Lemma}

Now choose $a$ sufficiently small and $a^\prime$ sufficiently large that inequalities (\ref{HbT}) and (\ref{HbT2}) hold.
Write $H^*_T(a)$ and $G^*_T(a^\prime)$ for the random variables corresponding (as in Lemma  \ref{L:indePoi}) to $H_T(a)$ and   $G_T(a^\prime) $ 
defined in terms of the times in the coupled EA process.
Consider the event
 \begin{equation}
 \{\Tcouple > T, \ H^*_T(a)  < T/4,   G^*_T(a^\prime) < T/12,  \widetilde{N}_b(T) \ge T/3 \} .
\label{Bevent}
\end{equation} 
On this event, take the ``good" intervals comprising  $\widetilde{N}_b(T)$ (with total length $\ge T/3$) and delete the intervals 
comprising $G^*_T(a^\prime)$ (with total length $< T/12$).
There remain ``good" intervals with total length $>T/4$, so there are at least $\lfloor aT \rfloor$ such intervals.
In other words, on  event (\ref{Bevent})   there are
at least $\lfloor aT \rfloor$ steps $i$ of the coupled process such that
\begin{equation}
 [\tau_i, \tau_{i+1}) 
\mbox{ is disjoint from } 
B^1 \cup B^2 \mbox{ and } \tau^+_i \le \tau_i + a^\prime .
\label{suchstep}
\end{equation}
For each such step $i$ we have (see argument below),
\begin{equation}
e^{- \tau_i/2} 
\max( \diam(C^1_i), \diam (C^2_i) ) 
\le b e^{a^\prime/2} := \beta, \mbox{ say}.
\label{tCC2}
\end{equation} 
In other words, on the event (\ref{Bevent})  we have
$N_\beta(T) \ge  \lfloor aT \rfloor$. 
So now
\begin{eqnarray*}
\lefteqn{ \Pr(\Tcouple > T, \ N_\beta(T) \le aT)} \\
&  \le  &
\Pr( \Tcouple > T, \ H^*_T(a)  > T/4) \\
&+& \Pr( \Tcouple > T, \ G^*_T(a^\prime)  > T/12)
+ \Pr( \Tcouple > T, \  \widetilde{N}_b(T) < T/3 ) \\
&\le& \Pr( H_T(a)  > T/4) + \Pr (G^*_T(a^\prime)  > T/12)
+ \Pr( \Tcouple > T, \  \widetilde{N}_b(T) < T/3 )
\end{eqnarray*}
using Lemma \ref{L:indePoi}.
Each term on the right has vanishing probability, by  (\ref{TcTv}) and (\ref{HbT})  and (\ref{HbT2}), and this establishes  Proposition \ref{P:diamcoup} 
(with $\beta$ in place of $b$).

\setlength{\unitlength}{0.1in}
\begin{figure}
\caption{The times $\tau_i$ of steps of the coupled process are shown on the axis.
The arrows point to the  times  
$\tau^1_{(j)}$ and $\tau^2_{(i-j)}$ associated with the completed steps 
of the component processes.
}
\label{Fig:tau2}
\begin{picture}(53,20)(-2,-10)
\put(-1,0){\line(1,0){53}}
\put(0,0){\circle*{0.6}}
\put(0,5){\circle*{0.6}}
\put(4,0){\circle*{0.6}}
\put(4,-5){\circle*{0.6}}
\put(10,0){\circle*{0.6}}
\put(10,5){\circle*{0.6}}
\put(17,0){\circle*{0.6}}
\put(17,5){\circle*{0.6}}
\put(20,0){\circle*{0.6}}
\put(20,-5){\circle*{0.6}}
\put(25,0){\circle*{0.6}}
\put(25,-5){\circle*{0.6}}
\put(32,0){\circle*{0.6}}
\put(32,5){\circle*{0.6}}
\put(34,0){\circle*{0.6}}
\put(34,-5){\circle*{0.6}}
\put(37,0){\circle*{0.6}}
\put(37,-5){\circle*{0.6}}
\put(42,0){\circle*{0.6}}
\put(42,5){\circle*{0.6}}
\put(0,0){\vector(0,1){4.2}}
\put(0,0){\vector(4,-5){3.5}}
\put(10,0){\vector(0,1){4.2}}
\put(10,0){\vector(2,-1){9.2}}
\put(17,0){\vector(0,1){4.2}}
\put(17,0){\vector(3,-5){2.6}}
\put(32,0){\vector(0,1){4.2}}
\put(32,0){\vector(2,-5){1.6}}
\put(42,0){\vector(0,1){4.2}}
\put(4,0){\vector(0,-1){4.2}}
\put(4,0){\vector(6,5){5.2}}
\put(20,0){\vector(0,-1){4.2}}
\put(20,0){\vector(12,5){11.2}}
\put(25,0){\vector(0,-1){4.2}}
\put(25,0){\vector(7,5){6.2}}
\put(34,0){\vector(0,-1){4.2}}
\put(34,0){\vector(8,5){7.2}}
\put(37,0){\vector(0,-1){4.2}}
\put(37,0){\vector(1,1){4.2}}
\put(42,0){\vector(1,-1){2.2}}
\put(-3,-0.5){$\tau_i$}
\put(-3,4.5){$\tau^1_i$}
\put(-3,-5.5){$\tau^2_i$}
\put(20.3,-1){$\tau_k$}
\put(25.3,-1){$\tau_{k+1}$}
\put(17.5,4.8){$\tau^1_{(j-1)}$}
\put(32.5,4.8){$\tau^1_{(j)}$}
\put(18.2,-6.3){$\tau^2_{(k-j)}$}
\put(23.2,-6.3){$\tau^2_{(k-j+1)}$}
\end{picture}
\end{figure}
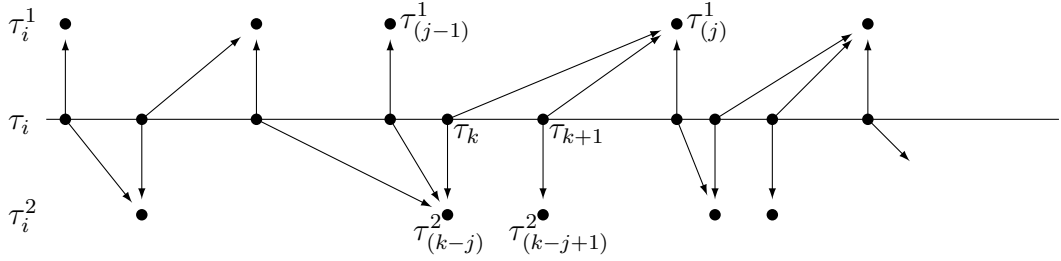

The argument that (\ref{suchstep}) implies (\ref{tCC2}) is 
illustrated by the case  in Figure \ref{Fig:tau2}.  
Consider $\tau_k = \min(\tau^1_k,\tau^2_k) 
= \min(\tau^1_{(j)},\tau^2_{(k-j)})$ 
for some $j$, and suppose that (as in the Figure) $\tau^1_{(j)} > \tau^2_{(k-j)}$.
Saying that 
\[
 [\tau_k, \tau_{k+1}) 
\mbox{ is disjoint from } 
B^1 \cup B^2 
\]
is saying that
\[   [\tau^1_{(j-1)}, \tau^1_{(j)}) 
\mbox{ is not in } 
B^1, \mbox{ and }
[\tau^2_{(k-j)}, \tau^2_{(k-j+1)}) 
\mbox{ is not in } 
B^2
\]
which is saying that
\[ e^{-\tau^1_{(j)}/2 } \diam(C^1_k) \le b
\mbox{ and }
e^{-\tau^2_{(k-j+1)}/2 } \diam(C^2_{k+1}) \le b .\]
Now consider the first time after $\tau_k$ that both components have expanded, that is 
\[
\tau^+_k := \min \{\tau_j : \tau^1_j > \tau^1_k \mbox{ and } \tau^2_j > \tau^2_k \} .
\]
Then the inequality above implies
\[
e^{-\tau^+_k/2} \max( \diam(C^1_k), \diam(C^2_k)) \le b . 
\]
So when $\tau^+_k \le \tau_k + a^\prime$ we have (\ref{tCC2}).

\subsection{Proof of Proposition \ref{Pcouple}}
\label{sec:ppcouple}
We need the following standard martingale-type bound. 
\begin{Lemma}
\label{L:SF}
Let $S \ge 1$ be a stopping time for a filtration $(\FF_n)$.  
For any $0<p_0<1$ and $m \ge 1$,
\[ \Pr(S > n) \le (1-p_0)^m + \Pr( L(n,p_0) < m, S > n) \]
where
\[ L(n,p_0) = \vert \{i: 1 \le i \le n, \ \Pr(S=i|\FF_{i-1})  \ge p_0 \} \vert .\]
\end{Lemma}
\begin{proof}
The process $(M_n)$ with $M_0 = 1$ and, for $n \ge 1$,
\begin{eqnarray*}
M_n =  & \quad 0 & \mbox{ on } \{S \le n\} \\
\quad = & \frac{1}{\prod_{1 \le i \le n} \Pr(S > i|\FF_{i-1})} &  \mbox{ on } \{S > n\}
\end{eqnarray*}
is a martingale.  On the event 
$\{  L(n,p_0) \ge m, S > n \}$ we have $M_n \ge (1-p_0)^{-m}$
and so
\[ 1 = \Ex M_n \ge \Ex M_n \indic_{(L(n,p_0) \ge m, S > n)} 
\ge (1-p_0)^{-m} \Pr( L(n,p_0) \ge m, S > n) .\] 
Because 
\[ \Pr(S > n)  =  \Pr( L(n,p_0) \ge m, S > n) +  \Pr( L(n,p_0) < m, S > n) \]
the result follows.
\end{proof}

We can now combine previous ingredients to prove Proposition \ref{Pcouple}.
Suppose $i$ is such that 
$\Icouple > i$ and the configuration 
$(\bC^1_i; \bC^2_i)$ satisfies (\ref{tCC}). 
Then by (\ref{pitz}) the probability of coalescing on the next step is 
at least
\[
\exp(- 4\pi  (|| z^1_0 - z^2_0 || e^{-\tau_i /2} +2 b)^2)
- 2c_0 || z^1_0 - z^2_0 || e^{-\tau_i /2}  .
\]
So there exist constants $\alpha > 0$ and $p_0 > 0$  
(determined by $b$ and $c_0$) such that, for the natural 
filtration $(\FF_i)$ of the coupled EA process,
\begin{equation}
 \Pr(\Icouple = i+1 \vert \FF_i) \ge p_0 
\mbox{ on } 
\{ \Icouple > i, \ || z^1_0 - z^2_0 || e^{-\tau_i /2} \le \alpha, 
\mbox{ $(\bC^1_i; \bC^2_i)$ satisfies (\ref{tCC})}
 \}
. \label{on3}
\end{equation}
Appealing to Lemma \ref{L:SF},
\begin{equation}
 \Pr(\Icouple > n) \le (1-p_0)^m + \Pr( L_n < m, \Icouple > n) 
 \label{Icn}
 \end{equation}
where
\[ L_n = \vert \{i: \ 0 \le i \le n-1, \ \Pr(\Icouple =i+1|\FF_{i})  \ge p_0 \} \vert .\]
Recall the definition of $N_b(t)$ in Proposition \ref{P:diamcoup}.
Take $a > 0$ (to be specified later) and consider some $m < an$.
If
\begin{equation}
 || z^1_0 - z^2_0 || e^{-\tau_j /2} \le \alpha \mbox{ for } 
j =  \lfloor an \rfloor - m 
\label{zet}
\end{equation}
then, on the event  $\{  \Icouple > n, \ N_b(t) > an \}$, the events in (\ref{on3}) hold for at least 
$m$ values of $i \le n$, which implies $L_n \ge m$.
So if (\ref{zet}) holds then
\[ \Pr(L_n < m, \Icouple > n) \le \Pr( \Icouple > n, N_b(n) \le an)  \]
and then using (\ref{Icn}) we have
\[  \Pr(\Icouple > n) \le (1-p_0)^m + \Pr( \Icouple > n, N_b(n) \le an) + 
 \Pr(\Icouple > n,  \mbox{ event (\ref{zet}) fails}) 
. \]
Proposition \ref{P:diamcoup} implies that
for sufficiently large $b$ there exist  constants $a, \rho > 0$ and $K< \infty$ such that
\[ \Pr(\Tcouple > n/3, \ N_b(n) \le an) \le K \exp(-\rho n), \ n = 1, 2, 3, \ldots  .\]
Using this choice of $a$ above, 
\[  \Pr(\Icouple > n) \le (1-p_0)^m + 
 \Oexp(n) 
\] \[
+   \Pr( \Icouple > n,   \Tcouple \le n/3)  + \Pr(\Icouple > n,  \mbox{ event (\ref{zet}) fails}) 
 \]
 where $\Oexp(n)$ denotes a ``vanishing probability" sequence which is $O( \rho^n)$ as $n \to \infty$ for some $\rho < 1$. 
 Now  note that elementary large deviation bounds for the rate-$2$ PPP $(\tau_i)$ show that 
 \begin{equation}
 \Pr( \Icouple > n,   \Tcouple \le n/3), \  \Pr( \Icouple \le  n,   \Tcouple > n), \ \Pr(\tau_n \le n/3) 
 \mbox{ are all } \Oexp(n).
 \label{eq41}
 \end{equation}
Choosing $m = \lfloor an/2 \rfloor$  we find
\[  \Pr(\Icouple > n) \le 
\Oexp(n) +
 \Pr( \mbox{ event (\ref{zet}) fails}) , \  n  = 1,2, 3, \ldots
 \]
where the $\Oexp(n)$ term does not depend on 
 $|| z^1_0 - z^2_0 ||$.
From the definition of event (\ref{zet}) and the choice of $m$
\begin{equation}
  \Pr( \mbox{ event (\ref{zet}) fails}) 
\le
\Pr( \tau_j \le 2 \log \sfrac{|| z^1_0 - z^2_0 ||} {\alpha} ) 
\mbox{ for } j = \lfloor an/2 \rfloor - 1 .
\label{tlog}
\end{equation}
From the final term in (\ref{eq41}) there exists a constant $\beta$ such that 
\[
  \Pr( \mbox{ event (\ref{zet}) fails}) = \Oexp(n) \mbox{ for } n > \beta  \log^+ || z^1_0 -  z^2_0 || .
  \] 
  So now 
  \[  \Pr(\Icouple > n) \le  \Oexp(n) \ \mbox{ for } n > \beta  \log^+ || z^1_0 -  z^2_0 || . \]
Because 
$\Pr( \Tcouple > n) \le  \Pr(\Icouple > n) + \Pr( \Icouple \le  n,   \Tcouple > n)$ 
we have established Proposition \ref{Pcouple}.

 \section{Proof of the main theorems}
 \subsection{Notation}
 \label{sec:Nota}
 In this section we use the preceding bounds to prove Theorems \ref{T:process} and \ref{T:limit}. 
 Recall some definitions from section \ref{sec:outline}.
 $\MM(\Reals^2)$ denotes the space of finite measures on $\Reals^2$, equipped with the usual topology of weak convergence. 
 For a particle $\xi$, the time-$t$ ancestor is denoted $\ancestor(t,\xi)$, and 
 $\descend(t_1,t_2,\zeta)$ denotes the set of particles born before $t_2$ whose time-$t_1$ ancestor is $\zeta$.
 And for $t_1  \le t_2$ and $\zeta \in \Xi_{\le t_1}$ 
\begin{eqnarray}
 \mbox{ $\mu_{t_1,t_2,\zeta}$ 
is the measure $\mu$ putting weight $e^{-t_2}$} \nonumber\\
\mbox{  on the position of each particle in $\descend(t_1,t_2,\zeta)$.} \label{def:mu3}
\end{eqnarray}
Note that given $\Xi_{\le t_1}$, the
``marks"  $(\mu_{t_1,t_2,\zeta}, \ \zeta \in \Xi_{\le t}) $ are still  {\em random} elements of the space  $\MM(\Reals^2)$, 
whose distributions depend on all $\Xi_{\le t_1}$ and are dependent as $\zeta$ varies.

If we use $\Xi_{\le t}$ to define a translation-invariant marked PPP of the form $\{(\xi, m^+(\xi)), \xi \in \Xi_{\le t}\}$ 
with non-negative real marks $m^+(\xi)$, then  there is a 
spatial average rate of mark values,  which we will write as
\[  \ave(m^+(\xi): \xi \in \Xi_{\le t}) \]
defined as the value of $a$ such that
 \[ \Ex \sum_{\xi \in \Xi_{\le t}, \ z_\xi \in B} m^+(\xi) =  a \times \area(B) ,   \quad B \subset \Reals^2 . \]
For instance we have, for $t_1 < t_2$, 
\begin{equation}
\ave(| \mu_{t_1,t_2,\xi}|: \xi \in \Xi_{\le t_1}) = 1 
\label{ave=1}
\end{equation}
where $|\mu|$ denotes the total mass of $\mu$.

The self-similarity property of the underlying space-time PPP $\bm{\Xi}$ allows us to write down exact 
self-similarity properties for our marked point processes. 
In particular, the action of the scaling map $z \to e^{-t/2} z$ on $\Reals^2$, applied to the distribution 
of $\{(z_\xi, \mu_{t_1,t_2,\xi}), \ \xi \in \Xi_{\le t_1} \}$, 
gives a distribution which coincides with the distribution obtained from 
$\{(z_\xi, \mu_{t_1+ t, t_2+t,\xi}), \ \xi \in \Xi_{\le t_1 +t} \}$ under the action of rescaling weights
$\mu \to  e^{-t} \mu$.
These self-similarity properties allow us to take previous
results, which were stated in the context of time decreasing from $0$ to $-t$, and rewrite them
in the context of time decreasing from $t$ to $0$ and in the notation above.  
These rewritten results and simple consequences are recorded as Corollaries \ref{C:distz} -- \ref{C:3} below.

As a first example, Proposition \ref{P1} shows that for $t > 0$
\[ \ave( ||z_{\ancestor(-t, \xi)} - z_\xi|| \ : \ \xi \in \Xi_{\le 0} ) \le K e^{t/2} \]
where $K = \int_0^\infty G(r) dr < \infty$.  
Using self-similarity this implies
\begin{Corollary}
\label{C:distz}
For $0 \le t_0 \le t$,
\[ \ave( ||z_{\ancestor(t_0, \xi)} - z_\xi|| \ : \ \xi \in \Xi_{\le t} ) \le K e^t e^{-t_0/2} .\]
\end{Corollary}
Next, the fact that a set of cardinality $k$ contains $k(k-1)$ ordered 
distinct pairs gives the first identity below, and the second follows from self-similarity. 
For $t > 0$
\begin{eqnarray}
\lefteqn{
\ave ( |\mu_{0,t,\zeta}| ( |\mu_{0,t,\zeta}| -e^{-t}) \ : \zeta \in \Xi_{\le 0} )
}
&& \label{43} \\
&=&
e^{-2t} 
\ave \left( \sum_{\xi_2 \in \Xi_{\le t}} \indic_{ \{ \ancestor(0,\xi_2) = \ancestor(0,\xi_1) \} } : \xi_1 \in \Xi_{\le t} \right)\nonumber  \\
&=&
e^{-t}
\ave \left( \sum_{\xi_2 \in \Xi_{\le 0}} \indic_{ \{ \ancestor(-t,\xi_2) = \ancestor(-t,\xi_1) \} } : \xi_1 \in \Xi_{\le 0} \right)  \nonumber \\
&=&
e^{-t} \int_\square \int_{\Reals^2} q(t; z_1,z_2) dz_2 dz_1 
= e^{-t} \int_{\Reals^2} q(t; \origin,z) dz \nonumber 
\end{eqnarray}
where $\square$ denotes the unit square and $q(t; z_1,z_2)$ is the probability, given that $\Xi_{\le 0}$ has particles at $z_1$ and $z_2$, that they have the same ancestor at time $-t$.
In the notation of Proposition \ref{P1} we have (using the triangle inequality)
$q(t; \origin,z) \le 2 G_t( e^{-t/2} || z ||/2)$. 
So from the conclusion of Proposition \ref{P1} we have
\[ e^{-t} \int_{\Reals^2} q(t; \origin,z) dz 
\le 2 e^{-t} \int_{\Reals^2} G( e^{-t/2} || z ||/2) dz 
= 2 \int_{\Reals^2} G( || z ||/2) dz < \infty \]
Using (\ref{ave=1}) for the $-e^t$ term in (\ref{43}) we have established
\begin{Corollary}
\label{C:tL2}
$\sup_{t > 0} \ave ( |\mu_{0,t,\zeta}|^2 \ : \zeta \in \Xi_{\le 0} ) < \infty  $.
\end{Corollary}

Finally, self-similarity allows us to rewrite Proposition \ref{Pcouple}  as follows.
\begin{Corollary}
\label{C:3}
Let $s_0  \le \bar{s} \le s_1 \le s_2$.
Let $p(s_0,s_1,s_2; z_1,z_2)$ be the probability, given that $\Xi_{\le s_1}$ contains a particle at position $z_1$ 
and $\Xi_{\le s_2}$ contains a particle at position $z_2$, that these particles have different time-$s_0$ ancestors. 
Then
\[ p(s_0, s_1,s_2; z_1,z_2) \le K \exp(- \rho (\bar{s} - s_0)) \mbox{ provided } ||z_1 - z_2|| \le \sqrt{2}  e^{-\bar{s}/2}  \]
for the constants $K, \rho$ in Proposition \ref{Pcouple} .
\end{Corollary}

\subsection{Convergence of mark measures}
\label{sec:MPPs}
Here we will prove
\begin{Proposition}
\label{P:coupledist}
There exist $\MM(\Reals^2)$-valued  marks 
$(\mu_{0, \infty, \zeta}, \ \zeta \in \Xi_{\le 0})$ such that 
$\ave ( | \mu_{0, \infty, \zeta}|, \ \zeta \in \Xi_{\le 0}) = 1$ and
for all $\zeta \in \Xi_{\le 0}$
\[  \mu_{0,t,\zeta} \to \mu_{0, \infty, \zeta} \mbox{ in probability as } t \to \infty.
\]
\end{Proposition}
The argument is slightly subtle: we cannot directly compare $\mu_{0,t_1,\zeta}$ and $\mu_{0,t_2,\zeta}$  because the relative numbers of time-$t_2$ descendants of different time-$t_1$ 
ancestors are different (as in a supercritical branching process), so the measure on time-$t_1$ descendants derived from the uniform measure on time-$t_2$ descendants is not uniform.
Instead we first prove convergence of total masses.

\begin{Proposition}
\label{P:couplenew}
There exist real-valued marks 
$(m_{0, \zeta}, \ \zeta \in \Xi_{\le 0})$ such that 
$\ave (m_{0, \zeta} : \ \zeta \in \Xi_{\le 0}) = 1$  and
for all $\zeta \in \Xi_{\le 0}$
\[ | \mu_{0,t,\zeta}| \to m_{0, \zeta} \mbox{ in probability as } t \to \infty.
\]
\end{Proposition}
\begin{proof}
Write $\square$ for the unit square and $\square(t)$ for the scaled square of area $e^{-t}$.
For $0 <t_0 <  t_1 < t_2$  and $\eps > 0$ write $A(t_0, t_1,t_2,\eps)$ for the event
\begin{quote}
there exist at least $(1 - \eps)e^{t_1-t_0}$ particles of $\Xi_{\le t_1}$ in $\square(t_0)$, and 
at least $(1 - \eps)e^{t_2-t_0}$ particles of $\Xi_{\le t_2}$ in $\square(t_0)$, 
all with the same time-$0$ ancestor.
\end{quote}
Define
\begin{equation}
 \rho(t_0, \eps) = \liminf_{t_1 \to \infty} \liminf_{t_2 \to \infty} \Pr(A(t_0,t_1,t_2,\eps)) . 
 \label{reps}
 \end{equation}
We will show
\begin{Lemma}
$\lim_{t_0 \to \infty} \rho(t_0, \eps)  = 1 
\mbox{ for each } \eps > 0 .$
\label{t00}
\end{Lemma}
Granted that, consider
\[
a(t_1,t_2) :=
 \ave( \ \min( \ | \mu_{0,t_1,\zeta}| \ , \ | \mu_{0,t_2,\zeta}|  \ ) \ , \ \zeta \in \Xi_{\le 0})  \ \le 1 .
\]
By averaging over area-$t_0$ squares in $\Reals^2$, 
\[ a(t_1,t_2) \ge (1 - \eps)  \Pr(A(t_0,t_1,t_2,\eps)) . \]
So Lemma \ref{t00} implies
\begin{equation}
\liminf_{t_1 \to \infty} \liminf_{t_2 \to \infty} a(t_1,t_2) = 1  .
\label{at1t2}
\end{equation}
By the triangle inequality and (\ref{ave=1}),
\[ \ave( \left| \ | \mu_{0,t_1,\zeta}| \ - \ | \mu_{0,t_2,\zeta}|  \  \right|  : \zeta \in \Xi_{\le 0}) 
\le 2 a(t_1,t_2) .
\]
Then (\ref{at1t2}) and the Cauchy criterion imply there exist limits $m_{0,\zeta} $ for which 
\[ \lim_{t \to \infty} \ave( \left| \ | \mu_{0,t,\zeta}| \ - \  m_{0,\zeta} \  \right|  : \zeta \in \Xi_{\le 0}) = 0 . \]
Finally, Corollary \ref{C:tL2} provides the ``uniform integrability" condition needed to pass (\ref{ave=1}) to the limit to obtain
$\ave (m_{0, \zeta} : \ \zeta \in \Xi_{\le 0}) = 1$.  
This establishes Proposition \ref{P:couplenew}.
\end{proof}

\begin{proof}

[of Lemma  \ref{t00}]
Write  $\Xi^{\square(t_0)}_{\le t}$ for the restriction of $\Xi_{\le t}$ to particles within $\square(t_0)$.
We can upper bound the mean number of pairs $(\xi_1, \xi_2)$ in $\square(t_0)$ with $\xi_i \in \Xi_{\le t_i}$ and with different time-$0$ ancestors, as follows.
Write
\[ M(t_0,t_1,t_2) := 
\sum_{\xi_1 \in \Xi^{\square(t_0)}_{\le t_1}}   \sum_{\xi_2 \in \Xi^{\square(t_0)}_{\le t_2}     }  
 \  \indic_{ \{ \ancestor(0,\xi_1) \neq   \ancestor(0,\xi_2)  \}} 
. \]
Then
\[
  \Ex M (t_0,t_1,t_2) 
= e^{t_1} e^{t_2} \ 
\int_{\square(t_0)} \int_{\square(t_0)} p(0,t_1,t_2,z_1,z_2) \ dz_1 dz_2
\]
where $p(0,t_1,t_2,z_1,z_2)$ is the probability, given that $\Xi_{\le t_1}$ has a particle at $z_1$ and $\Xi_{\le t_2}$ has a particle at $z_2$, 
that these two particles have different time-$0$ ancestors. 
But Corollary \ref{C:3} shows that when $z_1$ and $z_2$ are in $\square(t_0)$ we have
$ p(0,t_1,t_2,z_1,z_2) \le K \exp(- \rho t_0) $.
So 
\begin{equation}
 \Ex M(t_0,t_1,t_2) \le e^{t_1+t_2 - 2t_0} \ K \exp(- \rho t_0) . 
\label{Mttt}
\end{equation}

We now quote an elementary lemma.
\begin{Lemma}
\label{L:IJf}
Let $I \subset J$ be finite sets, let $\sim$ be an equivalence relation on $J$ and 
let $B$ be a maximal-cardinality set in the 
corresponding partition of $J$.  Let
\[ \rho = \frac
{ | \{ (i,j) \in I \times (J \setminus I) \  : \ i \not\sim j \}| }{|I| \ \cdot \ |J \setminus I|} . \]
Then 
$| B \cap I| \ge (1 - \rho) |I|$ and $| B \cap (J \setminus I) | \ge (1 - \rho)  |J \setminus I|$.
\end{Lemma}
We will apply the lemma with $I$ and $J$ being $\Xi^{\square(t_0)}_{\le t_1}$ and $\Xi^{\square(t_0)}_{\le t_2}$, so that 
$|I|$ and $|J|$ have Poisson distributions with means 
$e^{t_1-t_0}$ and $e^{t_2-t_0}$, and to the equivalence relation ``same time-$0$ ancestor". 
Choose $\delta > 0$ such that 
$(1 - \delta)(1 - \delta(1-\delta)^{-2}) < \eps$. 
On the event
\begin{equation}
|I| \ge (1 - \delta) e^{t_1-t_0} \mbox{ and }  |J| \ge (1-\delta) e^{t_2-t_0}  \mbox{ and }  M(t_0,t_1,t_2) \le \delta e^{t_1+t_2 - 2t_0}  
\label{ddd}
\end{equation}
Lemma \ref{L:IJf} implies that event $A(t_0, t_1,t_2,\eps)$ holds.  
The first two events in (\ref{ddd})  have probabilities $\to 1$ as $t_1, t_2 \to \infty$, and so by (\ref{Mttt}) and Markov's inequality
the limit $\rho(t_0,\eps)$ at (\ref{reps}) satisfies
\[ \rho(t_0,\eps) \ge 1 - \delta^{-1} K \exp(- \rho t_0) , \] 
establishing Lemma \ref{t00}.
\end{proof}

\paragraph{Proof of Proposition \ref{P:coupledist}.}
Take $0 < t_0 < t$.
By self-similarity, Proposition \ref{P:couplenew} remains true if time $0$ is replaced by an arbitrary time $t_0$: 
there exist real-valued marks 
$(m_{t_0, \xi}, \ \xi \in \Xi_{\le t_0})$ such that for all $\xi \in \Xi_{\le t_0}$
\begin{equation}
 | \mu_{t_0,t,\xi}| \to m_{t_0, \xi} \mbox{ in probability as } t \to \infty. 
 \label{eq:mutt}
 \end{equation}
For $\zeta \in \Xi_{\le 0}$ define $\nu_{0,t_0,\zeta}$ to be the measure that puts weight $ m_{t_0, \xi} $ 
on the position $z_\xi$ of each particle $\xi \in \descend(0,t_0,\zeta)$.
And define $\nu_{0,t_0,t,\zeta}$ to be the measure that puts weight $|  \mu_{t_0, t,\xi} |$ 
on the position $z_\xi$ of each particle $\xi \in \descend(0,t_0,\zeta)$.  
We will need to show that, for large $t_0$, the measures
 $\nu_{0,t_0,t,\zeta}$ and $\mu_{0,t,\zeta} $ are close.
 
 We exploit the dual bounded Lipschitz metric on $\MM(\Reals^2)$:
 \[ d(\nu, \nu^\prime) = \sup \left\{ \left| \int f d \nu - \int f d\nu^\prime \right|: \ ||f||_{BL} \le 1 \right\} \]
 \[ ||f||_{BL}  := \max \left( \sup_z |f(z)|, \ \sup_{z_1 \ne z_2} \sfrac{|f(z_2) - f(z_1)|}{ ||z_1 - z_2||} \right) .\]
 This metric has the property
 \begin{equation}
 d \left( c \sum_i \delta_{z_i}, c \sum_i \delta_{z^\prime_i} \right) \le c \sum_i ||z_i - z^\prime_i|| . 
 \label{eq:cd}
 \end{equation}
 Consider $0 < t_1 < t_2 < t$.
The relationship between $\nu_{0,t_2,t,\zeta}$ and $\nu_{0,t_1,t,\zeta}$ is that for  
each $\xi \in \descend(0,t,\zeta)$ the weight $e^{-t}$ moved from the position of 
$\ancestor(t_2,\xi)$ to the position of 
$\ancestor(t_1,\xi)$.
Taking spatial averages and using (\ref{eq:cd}) we find
\begin{eqnarray*}
\ave( d(\nu_{0,t_1,t,\zeta}, \nu_{0,t_2,t,\zeta}), \ \zeta \in \Xi_{\le 0} )  
&\le & e^{-t}  \ave( || z_{\ancestor(t_1,\xi)} - z_{\ancestor(t_2,\xi)} ||, \ \xi \in \Xi_{\le t}) 
\\
& \le  & 2 K e^{-t_1/2} \mbox{ by Corollary \ref{C:distz}}. 
\end{eqnarray*}
This and  (\ref{eq:mutt}) are sufficient to imply that the $\nu$'s have a limit: for all $\zeta \in \Xi_{\le 0}$
 \begin{equation}
   \nu_{0,t,\zeta} \to \mu_{0, \infty, \zeta} \mbox{(say),  in probability as } t \to \infty.
\label{eq:znm}
\end{equation}
 Now by (\ref{eq:mutt}) we can write the definition of $ \nu_{0,t_0,\zeta}$ as
\[  \nu_{0,t_0,\zeta} = \sum \{ (\lim_{u \to \infty} | \mu_{t_0,u,\xi}|) \delta_{z_\xi} \ : \ \xi \in \descend(0,t_0,\zeta) \}   \]
whereas (by definition) for $t > t_0$ 
\[ \mu_{0,t,\zeta}  = \sum \{  \mu_{t_0,t,\xi} \ : \ \xi \in \descend(0,t_0,\zeta) \}  . \]
So now we have
\begin{eqnarray*}
d(\mu_{0,t,\zeta} , \mu_{0, \infty, \zeta}) & \le & d(\nu_{0,t_0,\zeta} , \mu_{0, \infty, \zeta}) + d(\mu_{0,t,\zeta} , \nu_{0, t_0, \zeta}) \\
& \le & d(\nu_{0,t_0,\zeta} , \mu_{0, \infty, \zeta})  
+ \sum \{ d ( \lim_{u \to \infty} |\mu_{t_0,u,\xi}|  \ \delta_{z_\xi} \ , \ \mu_{t_0,t,\xi} ) :\ \xi \in \descend(0,t_0,\zeta) \} \\
& \le & d(\nu_{0,t_0,\zeta} , \mu_{0, \infty, \zeta})  
+ \sum \left\{ \left|   \lim_{u \to \infty} |\mu_{t_0,u,\xi}| \ - \ | \mu_{t_0,t,\xi}| \ \right| : \ \xi \in \descend(0,t_0,\zeta) \ \right\} \\
&& + \sum \{ d( \mu_{t_0,t,\xi} \ , \ |\mu_{t_0,t,\xi}| \delta_{z_\xi}) : \xi \in \descend(0,t_0,\zeta) \} .
\end{eqnarray*}
Taking the spatial average over $\zeta \in \Xi_{\le 0}$ of sums over all time-$t_0$ descendants of $\zeta$ is the same as 
taking the spatial average over all time-$t_0$ particles. 
So taking averages in the inequality above gives
\begin{equation}
\ave ( d(\mu_{0,t,\zeta} , \mu_{0, \infty, \zeta}) \ : \zeta \in \Xi_{\le 0} )   \le b_1(t_0) + b_2(t_0,t) + b_3(t_0,t) 
\label{dmtz}
\end{equation}
where
\begin{eqnarray*} 
b_1(t_0) &=& \ave ( d(\nu_{0,t_0,\zeta} , \mu_{0, \infty, \zeta})  \ : \ \zeta \in \Xi_{\le 0} ) 
\\
b_2(t_0,t) &=&
 \ave \left( \left|   \lim_{u \to \infty} |\mu_{t_0,u,\xi}| \ - \ | \mu_{t_0,t,\xi}| \ \right| \ : \ \xi \in \Xi_{\le t_0} \right)
\\
b_3(t_0,t) &=&
 \ave ( d( \mu_{t_0,t,\xi} \ , \ |\mu_{t_0,t,\xi}| \delta_{z_\xi}) \ : \ \xi \in \Xi_{\le t_0} ) .
\end{eqnarray*}
To prove Proposition \ref{P:coupledist} it is enough to prove
\begin{equation}
\mbox{$\ave ( d(\mu_{0,t,\zeta} , \mu_{0, \infty, \zeta}) : \zeta \in \Xi_{\le 0} ) \to 0$ as $t \to \infty$.}
\label{madm}
\end{equation}
We know $b_1(t_0) \to 0$ as $t_0 \to \infty$ by (\ref{eq:znm}). 
And $b_2(0,t) \to 0$ as $t \to \infty$ by Proposition \ref{P:couplenew},
and then by self-similarity $b_2(t_0,t) \to 0$ as $t \to \infty$ for all $t_0$.
Finally,
\begin{eqnarray*}
d( \mu_{t_0,t,\xi} \ , \ |\mu_{t_0,t,\xi}| \delta_{z_\xi}) 
& \le &
\int || z_\xi - z|| \ \mu_{t_0,t,\xi}(dz) \\
&=&  e^{-t} \sum \{ ||z_{\ancestor(-t_0, \chi)} - z_\chi || \ : \ \chi \in \descend(t_0,t,\xi) \}
\end{eqnarray*}
and so 
\begin{eqnarray*}
 b_3(t_0,t) &\le&
e^{-t} \ave ( ||z_{\ancestor(-t_0, \chi)} - z_\chi || \ : \ \chi \in \Xi_{\le t} ) \\
& \le  & K e^{-t_0/2} \mbox{ by Corollary \ref{C:distz}}. 
\end{eqnarray*}
Now taking limits in the inequality (\ref{dmtz}) establishes (\ref{madm}) and then Proposition \ref{P:coupledist}.

 \subsection{The random partition}
 We  will now show that a limit random measure $\mu_{0,\infty,\xi}$ in Proposition  \ref{P:coupledist}
 is in fact Lebesgue measure $\Lambda$ restricted to some random set.
 The fact that the $t \to \infty$ limit normalized empirical measure on $\ZZ_t$ is $\Lambda$ implies that 
 \[ \sum \{ \mu_{0,\infty, \zeta} \  : \zeta \in \Xi_{\le 0}  \} = \Lambda  \mbox{ a.s. } .\]
So the random measures $\mu_{0,\infty, \zeta}$ have random densities $f_\zeta(z), \ z \in \Reals^2$ 
 such that
 \[ \sum  \{ f_\zeta(z) \ : \zeta \in \Xi_{\le 0}  \} = 1 \ \forall z \mbox{ a.s. } \]
As $t \to \infty$ we have
 \[
 \Ex \{ \sum_{z_{\xi_1} \in \square} \sum_{z_{\xi_2} \in \square} e^{-2t} \indic_{ \{||z_{\xi_1} - z_{\xi_2}||   \le \delta\}} \ : \ 
 \xi_1 \in \Xi_{\le t},  \xi_2 \in \Xi_{\le t} \}
 \to \int_\square \int_\square  \indic_{ \{ ||z_1-z_2|| \le \delta\} } \ dz_1 dz_2  
 \]
 Now consider whether a pair $(\xi_1,\xi_2)$ have different time-$0$ ancestors; precisely, consider
 \begin{equation}
  \lim_{t \to \infty} 
 \Ex \{ \sum_{z_{\xi_1} \in \square} \sum_{z_{\xi_2} \in \square} e^{-2t} \indic_{ \{||z_{\xi_1} - z_{\xi_2}||   \le \delta\}}
 \ \indic_{ \{   \ancestor(0, \xi_1) \neq \ancestor(0,\xi_2) \} }     \ : \ 
 \xi_1 \in \Xi_{\le t},  \xi_2 \in \Xi_{\le t} \}
 \label{zzdiff}
 \end{equation}
 From the fact 
 $\mu_{0,t,\zeta} \to \mu_{0, \infty, \zeta}$
 the limit in (\ref{zzdiff})  equals
 \begin{equation}
   \Ex \  \int_\square \int_\square  \indic_{ \{ ||z_1-z_2|| \le \delta\} } \     \left(1 - \sum_{\zeta \in \Xi_{\le 0}   } f_\zeta(z_1) f_\zeta(z_2)    \right)              \ dz_1 dz_2  . 
   \label{eq39a}
   \end{equation}
 But consider the probability
 $p(0,t,t,z_1,z_2)$,  given that $\Xi_{\le t}$ has particles at $z_1$ and $z_2$, that
 they have  different time-$0$ ancestors.
 By Corollary \ref{C:3} with $\bar{s}$ defined by $\delta/\sqrt{2} = \exp(- \bar{s}/2)$,
 \[ \mbox{ if } ||z_2-z_1|| \le \delta \mbox{ then } p(0,t,t,z_1,z_2) \le K \delta^{2 \rho} \mbox{ for } t \ge \bar{s} . \]
 So the limit in (\ref{zzdiff}) also equals
 \begin{equation}
  \int_\square \int_\square  \indic_{ \{ ||z_1-z_2|| \le \delta\} } \ \lim_t p(0,t,t,z_1,z_2)  \ dz_1 dz_2 
 \le     K \delta^{2 \rho}   \int_\square \int_\square  \indic_{ \{ ||z_1-z_2|| \le \delta\} }   \ dz_1 dz_2  .
 \label{eq39b}
 \end{equation}
 For probability distributions $(a_i)$ and $(b_i)$ we have
 $1 - \sum_i a_ib_i \ge  1 - \max_i a_i$.
 Applying this to (\ref{eq39a}) and using inequality (\ref{eq39b}), 
 \[ \frac
 { \Ex \  \int_\square \int_\square  \indic_{ \{ ||z_1-z_2|| \le \delta\} } \     \left(1 - \max_{\zeta \in \Xi_{\le 0}   } f_\zeta(z_1)    \right)              \ dz_1 dz_2 }
 {  \int_\square \int_\square  \indic_{ \{ ||z_1-z_2|| \le \delta\} }  \ dz_1 dz_2 }
 \ \le \ K \delta^{2 \rho} . 
 \]
 Letting $\delta \downarrow 0$ we deduce that a.s.
 \[ \max_{\zeta \in \Xi_{\le 0}}   f_\zeta(z)   = 1 \mbox{ a.e. }. \]
 So defining
 \[ A(0, \zeta) = \{z: f_\zeta(z) = 1 \} \]
 and modifying on null sets, the random sets 
 $ \{ A(0,\zeta) : \zeta \in \Xi_{\le 0}   \} $
 form a partition of $\Reals^2$, and 
$\mu_{0,\infty,\zeta}$ is Lebesgue measure restricted to $A(0,\zeta)$.
So writing $\Lambda_A$ for Lebesgue measure restricted to $A$, we can rewrite Proposition \ref{P:coupledist} as follows, 
using self-similarity to extend from the time-$0$ case to the general time $t$ case.
\begin{Proposition}
\label{P:partition}
For each $- \infty < t < \infty$ there exists
a random partition 
$\{ A(t,\zeta) : \zeta \in \Xi_{\le t} \}$ of $\Reals^2$ into measurable sets such that 
for all $\zeta \in \Xi_{\le t}$
\[  \mu_{t,u,\zeta} \to \Lambda_{A(t,\zeta)} \mbox{ in probability as } u \to \infty.
\]
\end{Proposition}

\subsection{Completing the proofs}
Proposition \ref{P:partition} is essentially enough to prove Theorems \ref{T:limit} and \ref{T:process}.
As noted in the introduction, for fixed $t$ we can regard 
\[ \bZ^{(t)} = 
\{ (z_\xi, A(t,\xi)): \ \xi \in  \Xi_{\le t}  \}
\]
as a marked point process.
The fact that the evolution of the coloring process after time $t$, given $\ZZ_t$, does not depend on the arrival times of the particles in $\Xi_{\le t}$, 
means that $ \bZ^{(t)} $ is measurable with respect to the time-reversed filtration 
$\stackrel{\leftarrow} {\FF_{t}}$ at (\ref{f-back}).  
The statement in Theorem \ref{T:process} was that the process $ \bZ^{(t)} $
evolves in reversed time according to the rule:
\begin{quote}
during $[t, t - dt]$, for each $\xi \in \Xi_{\le t}$  delete $\xi$ 
(that is, remove the entry  $(z_\xi, A(t, \xi))$ )
 with probability $dt$; 
for each deleted particle $\xi$, let $\zeta$ be the nearest other particle,
and set  $A(t - dt, \zeta) = A(t, \zeta) \cup A(t,\xi)$.
\end{quote}
To see how this arises, fix large $T$ and for $- \infty < t \le T$ consider 
$\{ (z_\xi, \mu_{t,T,\xi}): \ \xi \in  \Xi_{\le t}  \}$ 
as a marked point process.  From  Lemma \ref{L:rev} (the thinning property of the PPP), in reversed time $t$ 
this evolves precisely as a ``coalescing measures process": 
\begin{quote}
during $[t, t - dt]$, for each $\xi \in \Xi_{\le t}$  delete $\xi$ 
(that is, remove the entry  $(z_\xi,  \mu_{t,T,\xi}) $)
 with probability $dt$; 
for each deleted particle $\xi$, let $\zeta$ be the nearest other particle,
and set  $ \mu_{t-dt,T,\zeta } =   \mu_{t,T,\zeta } +  \mu_{t,T,\xi }$.
\end{quote}
Taking the $T \to \infty$ limit given in Proposition \ref{P:partition}, we obtain the  former rule for the dynamics of $ \bZ^{(t)} $.
 
The other assertions of Theorem  \ref{T:process} hold by translation-invariance and  self-similarity of the 
underlying space-time PPP $\bm{\Xi}$.

\newpage
\section{Discussion}
\subsection{In what sense is this a tree process?}
\label{sec:tree_sense}
We have  used the  language of ancestors and descendants,  but otherwise have not really exploited the implicit tree structure of the colored point process construction.
If we draw the process as a random tree  in the plane, with edges drawn as line segments, it is clear from Figure \ref{fig_1} that edges 
sometimes cross, so we do not get a ``tree" in the usual sense.  
This suggests that, in the opening ``$k$ colors in the unit square" model, in the limit partition into $k$ colored regions,  
these regions are not necessarily connected. 
Figure \ref{Fig:notconn}  illustrates how this could happen. 
Simulations strongly indicate  that in fact a typical region is not connected but that only a very small proportion of its area is outside its largest connected component.

\begin{figure}[h!]
\caption{A possible realization of the tree in the unit square on the first $n = 11$ arriving points.  The edge from $1$ to $2$ is omitted, to show the
$k  = 2$ subtrees associated with the first two vertices.  At this stage  
the unit square is Voronoi-partitioned into 2 components according to whether 
the nearest vertex is $\circ$ or $\bullet$, and  the $\circ$  component is not connected.  We expect this disconnection to persist in the $n \to \infty$ limit.
}
\label{Fig:notconn}
\setlength{\unitlength}{0.2in}
\begin{picture}(6,11)(-9,-1.5)
\multiput(0,0)(1,1){7}{\circle*{0.3}}
\put(2,5){\circle{0.3}}
\put(5,2){\circle{0.3}}
\put(1,5){\circle{0.3}}
\put(6,2){\circle{0.3}}
\put(0,0){\line(1,1){6}}
\put(2,5){\line(1,-1){3}}
\put(2,5){\line(-1,0){1}}
\put(5,2){\line(1,0){1}}
\put(0.15,-0.4){2}
\put(1.15,0.6){4}
\put(2.15,1.6){5}
\put(3.15,2.6){7}
\put(4.15,3.6){9}
\put(5.15,4.6){10}
\put(6.15,5.6){11}
\put(1.8,4.2){1}
\put(0.8,4.2){6}
\put(4.8,1.2){3}
\put(5.8,1.2){8}
\put(-2,-1){\line(1,0){10}}
\put(-2,-1){\line(0,1){9}}
\put(-2,8){\line(1,0){10}}
\put(8,-1){\line(0,1){9}}
\end{picture}
\end{figure}
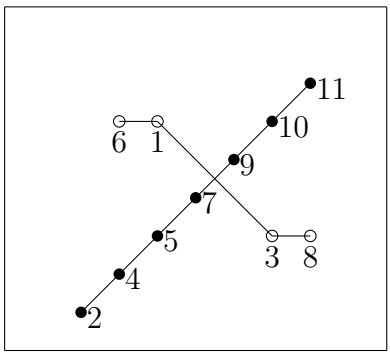

\newpage

 \subsection{Other models of coalescing partitions}
 \label{sec:other_coal}
 There has been very little study of partition-valued processes in the plane which evolve by merging of adjacent components.
 One such process can be obtained by thinning a Poisson {\em line} process, but we are thinking of pairwise mergers. 
A well-studied implicit example is provided by bond percolation.
As illustrated in Figure \ref{fig:bond}, 
to a percolation cluster of ``open" edges  (A) one can associate (this is {\em planar duality}) the region consisting of the union of the unit squares centered at the the vertices in the cluster (B), 
and then delete the open edges (C) and vertices to obtain a partition of the plane (D).  
The length of the boundary between two adjacent regions in this partition equals the number of ``closed"  edges between the original percolation clusters. 
So in the bond percolation model where edges become open at Exponential(1) times, 
the associated partition-valued process is such that the merger rate of adjacent regions equals the length of their common boundary. 
For this model classical percolation theory \cite{grimmett} implies that infinite regions appear after time $\log 2$.
 More general models in which two adjacent components merge into one at some stochastic rate determined by their geometry
are discussed in \cite{me-empires}, where it is conjectured that if large components do not grow too quickly (relative to small components) then there should be some self-similar
asymptotics, but no such result is proved.  
The coalescing partitions process in this paper is perhaps the only known self-similar Markovian process of pairwise merging partitions of $\Reals^2$.
In one dimension, the thinning process of Poisson points defines a self-similar process of merging adjacent intervals, which has an interpretation as intermediate-size 
asymptotics in the Kingman coalescent (\cite{me-coal} section 3.1).

\newpage

\setlength{\unitlength}{0.125in}

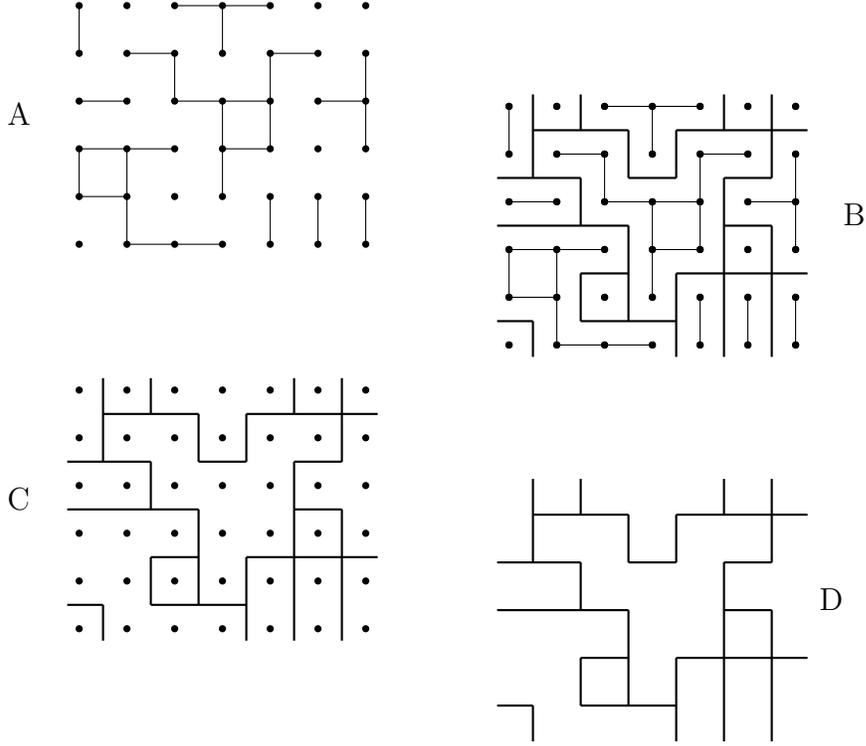
\begin{figure}[h!]
\caption{Bond percolation clusters as a partition of $\Reals^2$.}
\label{fig:bond}
\begin{picture}(12,12)
\thinlines
\multiput(0,0)(2,0){7}{\circle*{0.3}}
\multiput(0,2)(2,0){7}{\circle*{0.3}}
\multiput(0,4)(2,0){7}{\circle*{0.3}}
\multiput(0,6)(2,0){7}{\circle*{0.3}}
\multiput(0,8)(2,0){7}{\circle*{0.3}}
\multiput(0,10)(2,0){7}{\circle*{0.3}}
\put(2,0){\line(1,0){4}}
\put(0,2){\line(1,0){2}}
\put(0,4){\line(1,0){4}}
\put(6,4){\line(1,0){2}}
\put(0,6){\line(1,0){2}}
\put(4,6){\line(1,0){4}}
\put(10,6){\line(1,0){2}}
\put(2,8){\line(1,0){2}}
\put(8,8){\line(1,0){2}}
\put(4,10){\line(1,0){4}}
\put(0,2){\line(0,1){2}}
\put(0,8){\line(0,1){2}}
\put(2,0){\line(0,1){4}}
\put(4,6){\line(0,1){2}}
\put(6,2){\line(0,1){4}}
\put(6,8){\line(0,1){2}}
\put(8,0){\line(0,1){2}}
\put(8,4){\line(0,1){4}}
\put(10,0){\line(0,1){2}}
\put(12,0){\line(0,1){2}}
\put(12,4){\line(0,1){4}}
\put(-3,5){A}
\end{picture}

\begin{picture}(12,16)
\thicklines
\multiput(0,0)(2,0){7}{\circle*{0.3}}
\multiput(0,2)(2,0){7}{\circle*{0.3}}
\multiput(0,4)(2,0){7}{\circle*{0.3}}
\multiput(0,6)(2,0){7}{\circle*{0.3}}
\multiput(0,8)(2,0){7}{\circle*{0.3}}
\multiput(0,10)(2,0){7}{\circle*{0.3}}
\put(-0.5,1){\line(1,0){1.5}}
\put(3,1){\line(1,0){4}}
\put(3,3){\line(1,0){2}}
\put(7,3){\line(1,0){5.5}}
\put(-0.5,5){\line(1,0){5.5}}
\put(9,5){\line(1,0){2}}
\put(-0.5,7){\line(1,0){3.5}}
\put(5,7){\line(1,0){2}}
\put(9,7){\line(1,0){2}}
\put(1,9){\line(1,0){4}}
\put(7,9){\line(1,0){5.5}}
\put(1,-0.5){\line(0,1){1.5}}
\put(1,7){\line(0,1){3.5}}
\put(3,1){\line(0,1){2}}
\put(3,5){\line(0,1){2}}
\put(3,9){\line(0,1){1.5}}
\put(5,1){\line(0,1){4}}
\put(5,7){\line(0,1){2}}
\put(7,-0.5){\line(0,1){3.5}}
\put(7,7){\line(0,1){2}}
\put(9,-0.5){\line(0,1){7.5}}
\put(9,9){\line(0,1){1.5}}
\put(11,-0.5){\line(0,1){5.5}}
\put(11,7){\line(0,1){3.5}}
\put(-3,5){C}
\end{picture}

\begin{picture}(12,12)(-18,-24)
\thicklines
\multiput(0,0)(2,0){7}{\circle*{0.3}}
\multiput(0,2)(2,0){7}{\circle*{0.3}}
\multiput(0,4)(2,0){7}{\circle*{0.3}}
\multiput(0,6)(2,0){7}{\circle*{0.3}}
\multiput(0,8)(2,0){7}{\circle*{0.3}}
\multiput(0,10)(2,0){7}{\circle*{0.3}}
\put(-0.5,1){\line(1,0){1.5}}
\put(3,1){\line(1,0){4}}
\put(3,3){\line(1,0){2}}
\put(7,3){\line(1,0){5.5}}
\put(-0.5,5){\line(1,0){5.5}}
\put(9,5){\line(1,0){2}}
\put(-0.5,7){\line(1,0){3.5}}
\put(5,7){\line(1,0){2}}
\put(9,7){\line(1,0){2}}
\put(1,9){\line(1,0){4}}
\put(7,9){\line(1,0){5.5}}
\put(1,-0.5){\line(0,1){1.5}}
\put(1,7){\line(0,1){3.5}}
\put(3,1){\line(0,1){2}}
\put(3,5){\line(0,1){2}}
\put(3,9){\line(0,1){1.5}}
\put(5,1){\line(0,1){4}}
\put(5,7){\line(0,1){2}}
\put(7,-0.5){\line(0,1){3.5}}
\put(7,7){\line(0,1){2}}
\put(9,-0.5){\line(0,1){7.5}}
\put(9,9){\line(0,1){1.5}}
\put(11,-0.5){\line(0,1){5.5}}
\put(11,7){\line(0,1){3.5}}
\thinlines
\multiput(0,0)(2,0){7}{\circle*{0.3}}
\multiput(0,2)(2,0){7}{\circle*{0.3}}
\multiput(0,4)(2,0){7}{\circle*{0.3}}
\multiput(0,6)(2,0){7}{\circle*{0.3}}
\multiput(0,8)(2,0){7}{\circle*{0.3}}
\multiput(0,10)(2,0){7}{\circle*{0.3}}
\put(2,0){\line(1,0){4}}
\put(0,2){\line(1,0){2}}
\put(0,4){\line(1,0){4}}
\put(6,4){\line(1,0){2}}
\put(0,6){\line(1,0){2}}
\put(4,6){\line(1,0){4}}
\put(10,6){\line(1,0){2}}
\put(2,8){\line(1,0){2}}
\put(8,8){\line(1,0){2}}
\put(4,10){\line(1,0){4}}
\put(0,2){\line(0,1){2}}
\put(0,8){\line(0,1){2}}
\put(2,0){\line(0,1){4}}
\put(4,6){\line(0,1){2}}
\put(6,2){\line(0,1){4}}
\put(6,8){\line(0,1){2}}
\put(8,0){\line(0,1){2}}
\put(8,4){\line(0,1){4}}
\put(10,0){\line(0,1){2}}
\put(12,0){\line(0,1){2}}
\put(12,4){\line(0,1){4}}
\put(14,5){B}
\end{picture}

\begin{picture}(12,16)(-18,-24)
\thicklines
\put(-0.5,1){\line(1,0){1.5}}
\put(3,1){\line(1,0){4}}
\put(3,3){\line(1,0){2}}
\put(7,3){\line(1,0){5.5}}
\put(-0.5,5){\line(1,0){5.5}}
\put(9,5){\line(1,0){2}}
\put(-0.5,7){\line(1,0){3.5}}
\put(5,7){\line(1,0){2}}
\put(9,7){\line(1,0){2}}
\put(1,9){\line(1,0){4}}
\put(7,9){\line(1,0){5.5}}
\put(1,-0.5){\line(0,1){1.5}}
\put(1,7){\line(0,1){3.5}}
\put(3,1){\line(0,1){2}}
\put(3,5){\line(0,1){2}}
\put(3,9){\line(0,1){1.5}}
\put(5,1){\line(0,1){4}}
\put(5,7){\line(0,1){2}}
\put(7,-0.5){\line(0,1){3.5}}
\put(7,7){\line(0,1){2}}
\put(9,-0.5){\line(0,1){7.5}}
\put(9,9){\line(0,1){1.5}}
\put(11,-0.5){\line(0,1){5.5}}
\put(11,7){\line(0,1){3.5}}
\put(13,5){D}
\end{picture}
\end{figure}

\vspace*{-2.5in}

\subsection{Heuristic arguments}
\label{sec:2heuristics} 
Arguments in this section \ref{sec:2heuristics} are heuristics, only parts of which seem easily formalized.
We conjectured at (\ref{gamma:conj}) that the tail behavior of the ``meeting distance" random variable $|| \Zcouple ||$ 
is of the form 
\begin{equation}
\Pr( || \Zcouple || > r) \asymp r^{- \gamma} \mbox{ as } r \to \infty
\label{def:exponent}
\end{equation}
for some exponent $\gamma$.
This is heuristically related to the issue of fractal dimension of the boundaries of the 
regions $(A(0,\xi), \ \xi  \in \Xi_{\le 0})$, as follows.
Consider the  boundaries within the unit square.
Saying this has fractal dimension $d$ is saying that for small $x > 0$ we need order $x^{-d}$ radius-$x$ discs to cover these boundaries.  
Consider a uniform random point $z_1$ in the square and another random point $z_2$ uniform on $\disc(z_1,x)$.
The chance that $z_1$ and $z_2$  are in different components is the same order as the chance they are in the same covering disc, 
which chance is order $x^{2-d}$.  
But the former chance is the same order as the chance that the meeting distance $M_x$ between their lines of descent is at least  $1$, that is $\Pr(M_x > 1)$.  
So we expect $\Pr(M_x > 1) \asymp x^{2-d}$ as $x \downarrow 0$.
Now by self-similarity $M_x =_d x M_1$.  So setting $r = 1/x$
\begin{equation}
 \Pr(M_1 > r) \asymp r^{-(2-d)} 
 \label{def:exponent2}
\end{equation}
and we heuristically identify the fractal dimension as
\[
d = 2 - \gamma .
\]

 \subsubsection{Heuristic 1: the boundary has fractal dimension $1$}
\label{sec:istoofractal}

For large $t_0$ consider the Voronoi regions associated with the different sets of particles $\descend(0,t_0,\zeta)$ as $\zeta$ varies.
The particles in $\Xi_{\le t_0}$ are separated by distance of order $e^{-t_0/2}$.
Consider three points $(z_1, z_2, z_3)$ on the boundary at time $t_0$ at distances of (say) $5e^{-t_0/2}$ apart. 
As $t$ increases, particles arriving nearby move the boundary near these three points. 
The first such move 
is over a distance of order $e^{-t_0/2}$ and subsequent moves decrease geometrically.  
So in the $t \to \infty$ limit  the positions of the boundaries near $(z_1, z_2, z_3)$ should become 
$(z_1+ D_1 e^{-t_0/2}, z_2 + D_2 e^{-t_0/2} , z_3 + D_3 e^{-t_0/2})$ 
for some  $(D_1, D_2, D_3)$ with a non-vanishing limit as $t_0 \to \infty$.
But this is saying that in the limit partition $(A(0,\zeta), \zeta \in \Xi_{\le 0})$,
on every scale $\sigma = || y_1 - y_3||$, 
for $y_1$ and $y_3$ on the boundary, the distance from the midpoint $(y_1 + y_3)/2$ to the boundary is of the form $D \sigma$ for 
some random $D>0$.
This is a hallmark of ``fractal dimension $=1$".

\subsubsection{Heuristic 2: the boundary has fractal dimension $1$}
\label{sec:notfractal}
The genealogical tree defines a ``line of descent" for each particle in $\bm{\Xi}$; these particle positions are dense in $\Reals^2$, so
let us suppose that in the continuum limit there is such a ``line of descent" from almost all points $z$ of $\Reals^2$ to infinity.  
Draw the tree via line segments in $\Reals^2$.
Consider a point $(x,0)$ on the $x$-axis.
The route from there to infinity first crosses the $y = 1$ line at some random point $(c(x), 1)$.
Consider the random set $\CC$ of all such values $c(x)$ as $x$ varies.  
This is stationary (translation-invariant)  and so has some intensity $\gamma$, which cannot be zero;
moreover we expect  $\gamma < \infty$ because  the intensity of line segments of length $> a$ is finite for each $a > 0$.
 Then suppose that for each $c \in \CC$, the sets of originating points 
 $\{x: c(x) = c\}$ form some\footnote{The argument is
  unchanged if instead it is a union of a finite-mean number of intervals.}  interval $(x_-(c), x_+(c))$.

 Next consider, for $z > 0$, the random quantity defined as
 \begin{quote}
 the infimum of $y>0$ such that  the routes from $(x_1,0)$ to infinity and from $(x_1+z,0)$ to infinity 
 first hit the line $\{(x,y): - \infty < x < \infty \}$
 at the same point.
 \end{quote}
 This has a distribution, say $D_z$, which does not depend on $x_1$.
 By considering endpoints of the intervals $(x_-(c), x_+(c))$ we have 
 \[ \gamma = \lim_{\delta \downarrow 0} \Pr(D_\delta > 1)/\delta . \]
 But by scale-invariance we have $D_\delta =_d \delta D_1$, and so 
 \[ \Pr(D_1 > d) \sim \gamma/d \mbox{ as } d \to \infty . \] 
 But $D_1$ should have the same tail behavior as $M_1$ at (\ref{def:exponent2}). 
 This suggests the scaling exponent is $\gamma = 1$ and hence the fractal dimension of the boundaries $= 1$.

\subsubsection{Heuristic 3: the boundary has fractal dimension $\neq 1$}

The heuristics at (\ref{def:exponent},  \ref{def:exponent2}) are essentially saying that the fractal dimension $d$ is determined via the limit
\[ \lim_{r \to \infty} \frac{  \Pr( || \Zcouple || > 2r)  }{  \Pr( || \Zcouple || > r)  } = 2^{d - 2}  . \]
But the limit is determined by the asymptotics of the coupled EA process, conditioned on  not coalescing for a long time.
As we saw in in section \ref{sec:coupled} the dynamics of the coupled EA process involve the complicated geometry of excluded regions, and there seems
 no reason why that limit should turn out to be exactly $1/2$.

 \subsubsection{Regarding Conjecture \ref{C:1} }
One might imagine that Conjecture \ref{C:1} would follow easily from Proposition \ref{Pcouple} via some general result of the form
\begin{quote}
If $\{A, A^c\}$ is  a partition of the unit square such that $\rho(r) \to 0 $ as $r \to 0$,
where $\rho(r)$ is the probability that two random points at distance $r$ apart are in the same subset,
then (after modifying $A$ on a set of measure zero) the topological boundary of $A$ has measure zero.
 \end{quote}
 But this assertion is not true in general, by considering an example of the form 
 $A = \cup_i \disc(z_i,r_i)$ for dense $(z_i)$ and $r_i \downarrow 0$ very fast.
 Proving Conjecture \ref{C:1} seems to require some new argument.


\end{document}